\documentclass[12pt]{article}
\usepackage{preprint-layout} 
\usepackage{preprint-notation} 

\title{Conservative Discontinuous Cut Finite Element Methods }
\date{\today}

\author{
Mats G. Larson and  Sara Zahedi
}
\date{}

\newcommand{\tkapi}{\widetilde{\kappa}_i}

\begin{document}

\maketitle

\begin{abstract}
We develop a conservative cut finite element method for an elliptic coupled bulk-interface problem. The method is based on a discontinuous Galerkin framework where stabilization is 
added in such a way that we retain conservation on macro elements containing one element 
with a large intersection with the domain and possibly a number of elements with small intersections. We derive error estimates and present confirming numerical results.
\end{abstract}

\section{Introduction}
Simulations in many applications involve approximating solutions to Partial Differential Equations (PDEs) in complex geometries, such geometries may for example be defined by cell membranes, interfaces separating immiscible fluids, or heart valves guarding the exits of the heart cavities.  There has been an extensive development and progress of computational methods for efficiently approximating solutions to PDEs in complex geometries, see e.g.\cite{MiIa05, CaHuQuZa07, LiLoRaVo09, TrGr15, ElRa20}.  In the present work, we develop further one such computational technique, namely the \textbf{Cut Finite Element Method (CutFEM)}. 
In cut finite element methods the domain of interest is embedded into a polygonal domain equipped with a quasi-uniform mesh referred to as the fixed background mesh. An active mesh is defined which consists of all elements in the mesh that intersect the domain of interest. Associated to the active mesh is a finite dimensional function space and a weak formulation with  bilinear forms defined from the variational formulation of the PDE. Often consistent stabilization terms are added in the weak form to ensure stability and avoid ill-conditioned linear systems of equations. See \cite{BurClaHanLarMas15} for an introduction to CutFEM. 

In this paper we develop a cut finite element method based on a discontinuous Galerkin (DG) framework with local conservation in mind. We consider a model consisting of two subdomains separated by an interface with convection diffusion equations in the two bulk domains linearly coupled to a convection diffusion equation on the interface. The diffusion operators are in divergence form with variable tensor valued coefficients.
For a similar elliptic model problem with constant coefficients an unfitted finite element method based on the continuous Galerkin (CG) framework has been developed and analyzed in~\cite{GrOlRe15}. In~\cite{BurHanLarZah16} a CutFEM based on CG is proposed and analyzed for a coupled bulk-surface diffusion problem considering one bulk domain and later a CutFEM based on a discontinuous Galerkin (DG) framework was proposed in~\cite{Ma18}. For an early analysis of a finite element method based on fitted meshes for the coupled bulk-interface diffusion problem see~\cite{ElRa12}. In all the work described above stationary interfaces were considered. A space time CutFEM for simulations of coupled bulk-surface convection diffusion equations with moving interfaces, modelling for example the evolution of soluble surfactant concentrations, was developed in~\cite{HanLarZah16}. Cut finite element methods based on DG have also been proposed and analyzed for other problems for surface PDEs in e.g.~\cite{BuHaLaMa16} and bulk PDEs in~\cite{JoLa13, GuMA19}.

Compared to DG, the CG framework leads to finite element methods with fewer unknowns. An advantage of DG is its local conservation property~\cite{Do82,BaOD99,Cosh98}. In this work we develop a CutFEM based on DG that inherits this property which is important in many applications.  However, a straightforward application of the ghost penalty stabilization~\cite{Bu10} often used in CutFEM to ensure stability and avoid poor conditioning of the linear systems will destroy the local conservation. We propose and analyze a CutFEM based on DG were local conservation properties hold on so called macro elements.  We give a criteria for dividing elements in an active mesh into small and large elements. A small element is then connected via a chain of face neighbours to a large element and only on those faces stabilization~\cite{Bu10} is applied.  In this way, we create macro elements, consisting of a large element and possibly a few small elements. These macro elements behave as standard finite elements and local conservation properties from the DG framework are inherited to these macro elements. In this way we apply stabilization very restrictively but can prove that the macro element stabilization provides the same control as the full stabilization. In~\cite{LaZa19} we proposed a stabilization for continuous high order CutFEM for PDEs on interfaces where similar macro elements were created, but only as a tool in the proof. We note that the 
agglomeration techniques for discontinuous piecewise polynomial spaces developed in \cite{JoLa13} may be viewed as a strong version of the macro element stabilization we develop here. We also mention approaches for stabilization of continuous finite element spaces based on 
various extension techniques, see \cite{BaVe18,BuHa21, WaZaKrBe}. Note that all these references are restricted to bulk problems.

The paper is organized as follows. In Section 2 we introduce the model problem and its weak formulation. In Section 3 we formulate the proposed discontinuous cut finite element method and define the macro element stabilization and prove results on the properties provided by the stabilization. In Section 4 we analyze the method and prove optimal order a priori error estimates in the energy  and $L^2$ norms, and in Section 5 we show results from numerical experiments that support our theoretical findings. 
\section{The Model Problem}

\subsection{Basic Notation}
Let $\Omega \subset \IR^d$ be a $d$-dimensional domain, $d=2$ or $d=3$, with convex polygonal boundary $\partial \Omega$. Let $\Omega$ be partitioned into two subdomains $\Omega_1$ and $\Omega_2$ by a smooth closed $d-1$ dimensional internal interface $\Omega_0$.
Let $U_{\delta}(\Omega_0) = \cup_{x \in \Omega_0} B_\delta(x)$, where $B_\delta(x)$ is the open ball with radius $\delta>0$ 
centred at $x$, be the tubular neighborhood of $\Omega_0$ with thickness $\delta>0$. We assume that there is $\delta_0>0$ 
such that the closest point mapping $p:U_{\delta_0}(\Omega_0) \rightarrow \Omega_0$ is well defined and that 
$U_{\delta_0}(\Omega_0) \subset \Omega$, and thus the interface does not come arbitrarily close to the external 
boundary $\partial \Omega$.   Assume that $\Omega_2$ is the domain inside $\Omega_0$ and thus $\partial \Omega \subset \partial \Omega_1$. Let $n_i$ denote the exterior unit normal to  $\partial \Omega_i$, $i=1,2$.

\subsection{The Problem} 

Let $\nabla_i v= \nabla v$ be the $\IR^d$ gradient for the flat domains $\Omega_i$, $i=1,2,$ and let 
$\nabla_0 v = (I - n_0 \otimes n_0)\nabla v$ be the tangential gradient on the curved domain $\Omega_0$, where $n_0$ is a unit 
normal to $\Omega_0$, i.e., $n_0 = n_2$ or $n_0 = n_1|_{\Omega_0}$.  For $i=0,1,2,$ let $\beta_i$ be a smooth tangential vector field on $\Omega_i$ 
and assume that $\text{div}_i \beta_i=0$ in $\Omega_i$, where $\text{div}_i \beta_i$ is the divergence for $i=1,2,$ and the tangential 
divergence $\text{tr}(\beta_0 \otimes \nabla_0)$, for $i=0$. We also assume that  $n_i \cdot \beta = 0$ on $\partial \Omega_i$ for $i=1,2$.  For 
$i=0,1,2,$ let $A_i:\Omega_i \rightarrow \IR^{d\times d}$ be smooth uniformly positive definite matrix fields in the sense that there is 
a constant $\alpha_{0,i} >0$ and a function $\alpha_{0,i} \leq \alpha_i(x)$ such that for $x \in \Omega_i$,
\begin{equation}\label{eq:coer-Ai}
\alpha_{0,i}  \|\xi\|^2_{\IR^d}  \leq \alpha_i(x)\|\xi\|^2_{\IR^d} \leq (A_i \xi, \xi)_{\IR^d},  \qquad \xi \in T_x(\Omega_i) 
\end{equation} 
where $T_x(\Omega_i)$ is the tangent space at $x$ to $\Omega_i$ (which for $i=1,2,$ is $\IR^d$). Define the following operators 
\begin{equation}
\llbracket n \cdot A \nabla v \rrbracket = \sum_{i=1}^2 n_i \cdot A_i \nabla v_i,
\qquad 
[ \kappa v ]_i = \kappa_i v_i - \kappa_{0,i} v_0,  \quad i = 1,2
\end{equation}
where $\kappa_i\in \IR_+$, $\kappa_{0,i} \in \IR_+$, are constant parameters. We shall consider the following stationary 
convection-diffusion problem
\begin{alignat}{2}\label{eq:BVP}
  - \nabla \cdot (A_i \nabla u_i) +\nabla \cdot (\beta u_i)
  &= f_i &\qquad &\text{in $\Omega_i$, $i=1,2$}  
\\ \label{eq:BVP-b}
-\nabla_0 \cdot (A_0 \nabla_0 u_0) + \nabla_0 \cdot (\beta u_0)
+\llbracket n \cdot A \nabla u \rrbracket &=f_0 & \qquad &\text{on $\Omega_0$} 
\\ \label{eq:BVP-c}
-n_i \cdot \nabla A_i u_i &= [\kappa u ]_i  &\qquad &\text{on $\Omega_0$} 
\\ \label{eq:BVP-d}
u_1 & = 0  &\qquad &\text{on $\partial \Omega$} 
\end{alignat}
For simplicity, we consider the case of homogeneous Dirichlet data on the exterior boundary $\partial \Omega$ since we focus 
on the coupling at the interface and the extension to more general boundary conditions follows by standard techniques.

The model problem (\ref{eq:BVP})-(\ref{eq:BVP-d}) is obtained by restricting the general time dependent model describing the evolution of soluble surfactants with an interface, separating two immiscible fluids, moving with normal velocity $n \cdot \beta$, to an equilibrium state, see \cite{GrOlRe15} for a derivation. Note that in \cite{GrOlRe15}, the coefficients are scalars but here we allow matrix coefficients. Finally, we remark that if the boundary condition on the outer boundary $\partial \Omega$ is  replaced by a Neumann condition the right hand side must satisfy the equilibrium condition $\sum_{i=0}^2 \int_{\Omega_i} f_i = 0$.

\subsection{Weak Formulation}

To derive the weak form we let 
\begin{equation}
W=\left\{ v \in \oplus_{i=0}^2 H^1(\Omega_i) \text{ $|$ $v_1= 0$ on $\partial \Omega$} \right\} 
\end{equation} 
and note that $v\in W$ takes the form $v=(v_0,v_1,v_2)$ with components $v_i \in H^1(\Omega_i)$. Let for brevity
$\widetilde{\kappa}_i =\kappa_i \kappa_{0,i}^{-1}$,  multiplying (\ref{eq:BVP}) by the scaled test functions 
$\widetilde{\kappa}_i v_i$, and then using Green's formula, followed by the interface condition (\ref{eq:BVP-c}), and 
the equation on the interface (\ref{eq:BVP-b}), and finally using Green's formula on the interface we get
\begin{align}
\sum_{i=1}^2 (f_i,\tkapi v_i)_{\Omega_i} 
&=
\sum_{i=1}^2 -(\nabla \cdot A_i \nabla u_i, \tkapi v_i)_{\Omega_i}
+ (\nabla \cdot (\beta u_i), \tkapi v_i)_{\Omega_i}
\\
&
=\sum_{i=1}^2  \tkapi \underbrace{(  A_i \nabla u_i,\nabla v_i)_{\Omega_i}}_{a_i(u_i,v_i)} - (n_i \cdot A_i \nabla u_i,  \tkapi v_i)_{\partial \Omega_i\setminus \partial \Omega} 
\\
&\qquad 
+ \sum_{i=1}^2
\tkapi  \underbrace{\Big( \frac{1}{2} (\beta \cdot \nabla u_i,v_i)_{\Omega_i}  -\frac{1}{2}( u_i,\beta \cdot \nabla v_i)_{\Omega_i}\Big)}_{b_i(u_i,v_i)}
\\
&=\bigstar
\end{align}
\begin{align}
\bigstar &=
\sum_{i=1}^2 \tkapi a_i(u_i,v_i) + \tkapi  b_i(u_i,v_i)  - (\underbrace{n_i \cdot A_i \nabla u_i}_{=-[\kappa u]_i},\underbrace{\tkapi v_i - v_0}_{=\kappa_{0,i}^{-1}[\kappa  v]_i} )_{\Omega_0}
-
(n_i \cdot A_i \nabla u_i, v_0)_{\Omega_0}
\\
 &=
\sum_{i=1}^2  \tkapi a_i(u_i,v_i) +  \tkapi  b_i(u_i,v_i) 
+ (\kappa_{0,i}^{-1} [\kappa u]_i,[\kappa v]_i)_{\Omega_0}
- (\llbracket n \cdot A \nabla u \rrbracket,v_0)_{\Omega_0}
\\
&=
\sum_{i=1}^2  \tkapi  a_i(u_i,v_i) +  \tkapi  b_i(u_i,v_i) 
+ (\kappa_{0,i}^{-1}  [\kappa u]_i,[\kappa v]_i)_{\Omega_0}
- (f_0,v_0)_{\Omega_0}
\\
&\qquad 
- (\nabla_0 \cdot ( A_0 \nabla_0 u_0),v_0)_{\Omega_0}
+ (\nabla_0 \cdot (\beta u_0),v_0)_{\Omega_0}
\\
&=
\sum_{i=1}^2  \tkapi  a_i(u_i,v_i) +  \tkapi b_i(u_i,v_i) 
+ (\kappa_{0,i}^{-1} [\kappa u]_i,[\kappa v]_i)_{\Omega_0}
- (f_0,v_0)_{\Omega_0}
\\
&\qquad 
+ \underbrace{(A_0 \nabla_0 u_0, \nabla_0 v_0)_{\Omega_0}}_{a_0(u_0,v_0)}
+\underbrace{\frac{1}{2} (\beta \cdot \nabla_0 u_0,v_0)_{\Omega_0} -\frac{1}{2}(u_0,\beta \cdot \nabla_0 v_0)_{\Omega_0}}_{b_0(u_0,v_0)}
\end{align}
Thus, we arrive at the following weak formulation:

\paragraph{Weak Problem.} Find $u \in W$ such that 
\begin{equation}\label{eq:weakformu}
A(u,v) = L(v)\qquad \forall v \in W
\end{equation}
with forms defined, for $v,w \in W$, by 
\begin{align}
A(v,w) &= \sum_{i=0}^2  \tkapi \Big( a_i(v_i,w_i)+ b_i(v_i,w_i) \Big)
+  \sum_{i=1}^2 (\kappa_{0,i}^{-1} [\kappa v]_i,  [\kappa w]_i)_{\Omega_0}
\\
L(v) &= \sum_{i=0}^2 (f_i,\tkapi v_i)_{\Omega_i}
\end{align}
where $\widetilde{\kappa}_{0} = 1$, $\tkapi = \kappa_{0,i}^{-1} \kappa_i$ for $i=1,2$, and
\begin{align}
  a_i(v_i,w_i)&= ( A_i \nabla_i v_i,\nabla_i w_i)_{\Omega_i} \\
b_i(v_i,w_i) &= \frac{1}{2} \Big( (\beta \cdot \nabla_i v_i,w_i)_{\Omega_i} -( v_i,\beta \cdot \nabla_i w_i)_{\Omega_i} \Big)
\end{align}

\paragraph{Existence and Uniqueness.} Introducing the energy norm 
\begin{align}
\| v \|_A^2 =  A(v,v), \qquad v \in W
\end{align}
associated with the form $A$, we have the Poincar\'e inequality.
\begin{lem}  There is a constant such that for all $v \in W$,
\begin{equation}\label{eq:poincare-cont}
\sum_{i=0}^2 \tkapi \| v \|^2_{\Omega_i} \lesssim \| v \|^2_A
\end{equation}
\end{lem}
\begin{proof}
In order to show that (\ref{eq:poincare-cont}) holds, we let $\phi$ be the solution 
to the problem 
\begin{equation}
-\Delta \phi = \psi\quad \text{in $\Omega$}, \qquad \phi = 0 \quad \text{on $\partial \Omega$}
\end{equation}
where $\psi \in L^2(\Omega)$. Given $v \in W,$ multiplying $-\Delta \phi = \psi$ by 
$v = \sum_{i=1}^2 \tkapi v_i \chi_i$, where $\chi_i$ 
is the characteristic function of $\Omega_i$, and integrating by parts on the bulk domains 
$\Omega_i$, $i=1,2,$ we obtain
\begin{align}
&\sum_{i=1}^2 (\tkapi v_i,\psi)_{\Omega_i} = \sum_{i=1}^2 - (\tkapi v_i, \Delta \phi )_{\Omega_i} 
\\
&=  \sum_{i=1}^2 (\tkapi \nabla v_i,\nabla \phi )_{\Omega_i}  - (\tkapi v_i, \nabla_{n_i} \phi )_{\partial \Omega_i} 
\\
&=  \sum_{i=1}^2 (\tkapi \nabla v_i,\nabla \phi )_{\Omega_i}  - (\tkapi v_i - v_0, \nabla_{n_i} \phi )_{\Omega_0}  - ( v_0, \llbracket \nabla_{n_i} \phi \rrbracket )_{\Omega_0} 
\\
&=  \sum_{i=1}^2 (\tkapi \nabla v_i,\nabla \phi )_{\Omega_i}  - (\kappa_{0,i}^{-1}[\kappa v]_i, \nabla_{n_i} \phi )_{\Omega_0}  
\\
&\leq  \Big(\sum_{i=1}^2 \tkapi\| \nabla v_i\|^2_{\Omega_i} + \kappa_{0,i}^{-1}\|[\kappa v]_i\|^2_{\Omega_0} \Big)^{1/2} 
\Big( \sum_{i=1}^2 \tkapi \|\nabla \phi \|^2_{\Omega_i}  + \kappa_{0,i}^{-1}\|\nabla_{n_i} \phi \|^2_{\Omega_0}  \Big)^{1/2}
\\
&\lesssim \| v \|_A \| \psi \|_\Omega
\end{align}
where we used a trace inequality and elliptic regularity to obtain
\begin{equation}
\sum_{i=1}^2 \tkapi \|\nabla \phi \|^2_{\Omega_i}  + \kappa_{0,i}^{-1}\|\nabla_{n_i} \phi \|^2_{\Omega_0} 
\lesssim 
\underbrace{\max_{i\in \{1,2\}}\max(  \tkapi, \kappa_{0,i}^{-1})}_{=\max_{i\in \{1,2\}} \kappa_{0,i}^{-1} \max(1,\kappa_i)} 
\Big( \sum_{i=1}^2 \| \phi \|^2_{H^2(\Omega_i)} \Big)
 \lesssim\| \psi \|^2_\Omega
\end{equation}
Setting $\psi = \sum_{i=1}^2 v_i \chi_i$ gives
\begin{equation}\label{eq:contpoin-a}                                                                      
\sum_{i=1}^2 \tkapi \|v_i \|^2_{\Omega_i} \lesssim \| v \|^2_A
\end{equation}
with hidden constant dependent on  $\max_{i\in \{1,2\}} \kappa_{0,i}^{-1} \max(1,\kappa_i)$. Finally, using 
a trace inequality on $\Omega_2$ we get 
\begin{align}
\kappa_{2,0} \|v_0\|_{\Omega_0} &\lesssim  \| \kappa_{2,0} v_0 - \kappa_2 v_2 \|_{\Omega_0} + \kappa_2 \|v_2 \|_{\Omega_0} 
\\
&\lesssim  \| \kappa_{2,0} v_0 - \kappa_2 v_2 \|_{\Omega_0} + \kappa_2 \|v_2 \|_{H^1(\Omega_2)} 
\\
&\lesssim \| v \|_A
\end{align}
where we used (\ref{eq:contpoin-a}) to conclude that 
$\|v_2 \|^2_{H^1(\Omega_2)} \lesssim \| v_2 \|_A$.
This completes the proof of the Poincar\'e inequality (\ref{eq:poincare-cont}).
\end{proof}

Thanks to the Poincar\'e inequality, the energy norm $\| \cdot \|_A$ is indeed a norm and by definition $A$ is coercive 
and continuous with respect to $\| \cdot \|_A$ and we may apply the Lax-Milgram 
lemma to conclude that there exists a unique solution $u \in W$ to the weak problem (\ref{eq:weakformu}).

\section{Discontinuous CutFEM}
Here we formulate the discontinuous cut finite element method. We begin by introducing some preliminaries 
including the construction of the mesh and the finite element spaces. Then we formulate the method, define the stabilization forms, and provide some useful technical results for the stabilization forms. We end 
the section with a derivation of the method and the local conservation property.

\subsection{The Mesh and Finite Element  Spaces}
We introduce the following notation.
\begin{itemize} 
\item Let $\mcT_h$ be a quasiuniform partition of $\Omega$ into shape regular simplicies with mesh parameter $h \in (0,h_0]$ and 
let $\mcF_h$ be the set of internal faces in $\mcT_h$.

\item Define the active meshes 
\begin{equation}
\mcT_{h,i} = \{ T \in \mcT_h : T \cap \Omega_i \neq \emptyset\}, \qquad i=0,1,2
\end{equation}
associated with the subdomains $\Omega_i$ and let $\mcF_{h,i}$ be the set of interior faces in $\mcT_{h,i}$. The corresponding intersections with $\Omega_i$, are defined by
 \begin{equation}
\mcK_{h,i} = \{ K = T \cap \Omega_i : T \in \mcT_{h,i} \}, 
\qquad 
\mcE_{h,i} = \{ K = F \cap \Omega_i : F \in \mcF_{h,i} \}
\end{equation}

\item Let 
\begin{equation}
V_h = \{ v \in \oplus_{T \in {\mcTh}} P_1(T) \text{ $|$ $v=0$ on $\partial \Omega$}\}
\end{equation}
be the space of discontinuous piecewise linear polynomials on $\mcTh$, which are zero on the external boundary $\partial \Omega$. 
 Let $W_{h,i}= {V}_{h}|_{\mcT_{h,i}}$ be the active finite element space associated 
with $\mcT_{h,i}$ and let 
\begin{equation}\label{def:Wh}
W_h = \bigoplus_{i=0}^2 W_{h,i}
\end{equation}
be the finite element space associated with the full system.

\item 
For a face  $F$ in $\mcF_{h,i}$, $i=0,1,2$, shared by neighbouring elements $T_1$ and $T_2$ in $\mcT_{h,i}$ we let $\nu_{i,j}$ be the exterior unit normal vector to $T_j$.  For $i=0$,  let $\nu_{0,j}$ be the exterior unit co-normal to $K_j \in \mcK_{h,0}$, i.e. the vector which is tangent to $K_j$ and normal 
to $\partial K_j$. Define the jump and average operators at the face $F$ for scalar functions by
\begin{equation}
[ v ]= v_1 - v_2,
\qquad 
\langle v \rangle= \theta_{F,1} v_1 + \theta_{F,2} v_2
\end{equation}
where $v_l = v|_{T_l}$, $l=1,2$, and for  functions of the form $\nu_i \cdot v$, with $v$ a vector field, by
\begin{equation}
[ \nu_i \cdot v ]= \nu_{i,1} v_1 + \nu_{i,2} v_2, 
\qquad 
\langle \nu_i \cdot v \rangle= \theta_{F,1} \nu_{i,1} \cdot v_1 - \theta_{F,2} \nu_{i,2} \cdot v_2
\end{equation}
where  $ \theta_{F,j} \geq 0$ are weights defining a convex combination $\theta_{F,1} +  \theta_{F,2} =1$. The dual 
average $\langle \cdot \rangle^*$ is obtained by switching the weights in the average. Note that if the weights are equal, 
i.e., $ \theta_{F,1}= \theta_{F,2} = 1/2$ we have $\langle w \rangle^*$= $\langle w \rangle$. 

\item In particular, when $v = \beta \widetilde{v}$ 
for a smooth vector field $\beta$ and scalar $\widetilde{v}$ we write 
\begin{align}
[ \nu_i \cdot (\beta \widetilde{v} )  ] =  \nu_{i,1} \cdot \beta [ \widetilde{v} ] 
\end{align}
where we note that the right hand side is independent of the order of the enumeration of elements $T_1$ and $T_2$, 
 i.e. setting  index 2 to 1 and index 1 to 2, since $\nu_{i,2} = - \nu_{i,1}$ and changing the enumeration corresponds to 
 multiplying $[\widetilde{v}]$ by $-1$.
 
\item We have the following identity 
\begin{equation} \label{eq:jumpaverrel}
[ \nu_i \cdot v \, w ]=  \langle \nu_i \cdot v \rangle [w]+[ \nu_i \cdot  v] \langle w \rangle^*
\end{equation}
Here we note that changing the enumeration of $T_1$ and $T_2$,
corresponds to multiplying $ \langle \nu_i \cdot v \rangle $ and $[w]$ by $-1$ and therefore 
the product $\langle \nu_i \cdot v \rangle [w]$ is independent of the enumeration.

\end{itemize}

\subsection{The Method} 
\label{sec:cutfem}
The discontinuous cut finite element method takes the form: 
 find $u_h = (u_{h,0}, u_{h,1},u_{h,2}) \in W_h$ such that 
\begin{equation}\label{eq:weakformuh}
A_h(u_h,v) = L_h(v) \qquad \forall v \in W_h
\end{equation}
The forms are defined by 
\begin{align}
A_h(v,w) &= \sum_{i=0}^2 \tkapi \Big( a_{h,i}(v,w) +  b_{h,i}(v,w) +  s_{h,i}(v,w) \Big)
+ \sum_{i=1}^2 (\kappa_{0,i}^{-1} [\kappa v]_i,[\kappa w ]_i)_{\Omega_0}
\\
L_h(v) &=
 \sum_{i=0}^2 (f_i,\tkapi v_i)_{\Omega_i}
\end{align}
with  $\widetilde{\kappa}_{0} = 1$, $\tkapi = \kappa_{0,i}^{-1} \kappa_i$ for $i=1,2$, and the forms $a_{h,i}$ and $b_{h,i}$ defined by
\begin{align}
a_{h,i}(v_i,w_i)&
=( A_i \nabla_i v_i,\nabla_i w_i)_{\mcK_{h,i}} - (\langle \nu_i \cdot A_i \nabla_i v_i \rangle ,[w_i])_{\mcE_{h,i}} 
 \\
& \qquad 
- ([v_i],\langle \nu_i \cdot A_i \nabla_i w_i \rangle )_{\mcE_{h,i}}+ (\lambda_{a_i} h^{-1} [v_i], [w_i])_{\mcE_{h,i}}
 \\
 b_{h,i}(v_i,w_i)&=\frac{1}{2} \left( (\beta \cdot \nabla_i v_i,w_i)_{\mcK_{h,i}} -(v_i,\beta \cdot \nabla_i w_i)_{\mcK_{h,i}} \right)
\\
& \qquad 
+\frac{1}{2}\left(((\nu_i \cdot \beta)\langle v_i \rangle  ,[w_i])_{\mcE_{h,i}}  - ((\nu_i \cdot \beta)[v_i],\langle w_i \rangle )_{\mcE_{h,i}} \right) 
\\
&\qquad 
+ (\lambda_{b_i} [v_i], [w_i])_{\mcE_{h,i}}
%
\end{align}
where 
\begin{equation}\label{eq:weightnorm}
\lambda_{a_i}=\tau_{a_i} \| \nu_i \|^2_{A_i} = \tau_{a_i} \nu_i \cdot A_i \nu_i ,\qquad 
\lambda_{b_i} = \tau_{b_i} |\nu_i \cdot \beta|
\end{equation}
with parameters $\tau_{a_i}>0$ sufficiently large to guarantee coercivity, see Lemma \ref{lem:ahsh-coer} below,
and $\tau_{b_i} \geq 0$. 
The stabilization forms $s_{h,i}$, for $i=0,1,2$, are defined in (\ref{eq:stab-bulk}), (\ref{eq:stab-surface}), and (\ref{eq:stab-glob}) below and we derive the method 
in Section \ref{sec:derivation}.

%
%

\subsection{Definition of the Stabilization Forms}
\label{sec:stab}
We shall divide the elements in each active mesh $\mcT_{h,i}$ in two types, those with a large intersection with the domain $\Omega_i$ and 
those with a small intersection.  The small elements will be connected through a chain of face neighbours to a large element and together that 
set of elements form a macro element with a large intersection. The macro elements will essentially behave as standard finite elements in the bulk domains and for the surface domain 
we will have to add a stabilization in the direction normal to the surface since the partial differential equation at the surface only involves tangential derivatives.  We 
now make this approach precise with the following definitions.
\begin{itemize} 
\item An element $T$ in $\mcT_{h,i}$ has a large intersection if 
\begin{equation}\label{eq:largeel}
\gamma_i  \leq \frac{|T \cap \Omega_i|}{h_T^{d_i}} 
\end{equation}
where $h_T$ is the diameter of element $T$, $d_i$ is the dimension of the domain $\Omega_i$, 
and $\gamma_i$ is a positive constant which
 is independent of the element and the mesh parameter. The elements that are not large are defined to be small. Note that since the mesh is quasiuniform we have $h \sim h_T$ and it follows from (\ref{eq:largeel}) that 
 \begin{equation}
 h^{d_i} \lesssim |T \cap \Omega_i|
 \end{equation}
 
\item Let $\mcM_{h,i}$ be a macro element partition 
derived from $\mcT_{h,i}$, such that: each element $T \in \mcT_{h,i}$ belongs to precisely one $\mcT_{h,i,M}$, each macro 
element $M \in \mcM_{h,i}$ is a union of elements in $\mcT_{h,i,M}$, 
\begin{equation}
M = \cup_{T \in \mcT_{h,i,M}} T
\end{equation}
such that in $\mcT_{h,i,M}$ there is one element $T_M$ with a large intersection and all elements in $\mcT_{h,i,M}$ are 
connected to $T_M$ via a uniformly bounded number of internal faces in $\mcT_{h,i,M}$. 

\item Let $\mcF_{h,i}(M)$ be the set of interior faces in $\mcT_{h,i,M}$. Note that $\mcF_{h,i}(M)$ is empty when $M$ 
consist of only one element (with a large intersection). For each $M \in \mcM_{h,i}$ with $i=1,2,$ 
(the bulk domains)
we define the stabilization form 
\begin{equation}\label{eq:stab-bulk}
s_{h,i,M}(v,w) = \sum_{k=0}^1 \tau_{i,k} h^{2k - 1} ([\nabla^k v ], [\nabla^k w])_{\mcF_{h,i}(M)}
\end{equation}
and for $i=0$, (the surface domain)
\begin{equation}\label{eq:stab-surface}
s_{h,0,M}(v,w) =  \sum_{k=0}^1  \tau_{0,k}  h^{2k - 2}  ([\nabla^k v ], [\nabla^k w])_{\mcF_{h,0}(M)} 
+  \tau_{0,2} h^2 (\nabla_n v , \nabla_n w)_{M \cap \Omega_0}
\end{equation}
Here, $[\nabla^0 ]= [v ]$, $[\nabla^1 v ] = [\nabla v]$, and $\tau_{i,k}>0$ are positive parameters of the form 
\begin{equation}
\tau_{i,k} \sim c_{i,k} \| A_i\|_{L^\infty(M)}
\end{equation}
with $c_{i,k}>0$. The macro element stabilization control the full jump $[\nabla v]$ across faces internal to the macro element and the normal derivative $\nabla_{n} v$ at the interface. 

\item Given the macro element stabilization forms defined by (\ref{eq:stab-bulk}) and 
(\ref{eq:stab-surface}) we define the global stabilization forms
\begin{align}\label{eq:stab-glob}
s_{h,i}(v,w)&= \sum_{M \in \mcM_{h,i}} s_{h,i,M}(v,w) 
\end{align}
\end{itemize}

\begin{rem} Note that  \emph{there is no stabilization acting on the faces at the interface between two neighbouring 
macro elements}. In contrast, standard stabilization, which we refer to as full stabilization, includes stabilization terms also on those faces.
See Figure \ref{fig:a} for an illustration of a macro element partition 
and the edges on which stabilization is applied and Figure \ref{fig:illustfullvsmacrostab} for a comparison with full stabilization. 
\end{rem}

\begin{rem} The scaling, with respect to the mesh parameter $h$, in the different stabilization terms is chosen in such a way that 
the forthcoming coercivity estimate in Lemma \ref{lem:ahsh-coer}, which depend on the technical Lemma \ref{lem:macro-var}, holds. 
\end{rem} 

We next formulate a basic algorithm for computation of the macro element partition. We seek to generate 
macro elements that contain as few elements as possible, which corresponds to stabilization on as few faces as possible. Note, however, that the theoretical developments only requires a uniform bound on the number of elements in each macroelement and therefore it is not critical to find the optimal partition into macro elements. The algorithm computes a set 
\begin{equation}
\mcF^*_{h,i} = \cup_{M \in \mcM_{h,i}} \mcF_{h,i}(M) \subset \mcF_{h,i}
\end{equation}
containing all the faces where stabilization is applied.
 
\paragraph{Algorithm 1.} Initiate $\mcF_{h,i}^*$ as the empty set.
\begin{enumerate}
\item Mark each element in $\mcT_{h,i}$ as large or small according to criteria \eqref{eq:largeel}.
\item For every small element that is face connected to one or several large elements, choose 
one face that connects the small element to a large neighbouring element.  Add that 
face to $\mcF^*_{h,i}$ for stabilization and mark the small element as large. 
\item Repeat 2 until all elements are marked as large. 
\end{enumerate}
Note that the macro element partition and the generated set $\mcF^*_{h,i}$ is not unique since there is no preference in the choice of face in Step 2. For the coupled bulk-surface problem we use the algorithm described above to generate $\mcF^*_{h,0}$ from the mesh $\mcT_{h,0}$. To generate $\mcF^*_{h,i}$, for  $i=1,2$ we use the same algorithm but in the choice of face in Step 2, faces that are already  in  $\mcF^*_{h,0}$ are preferred. When the three sets $\mcF^*_{h,i}$, $i=0,1,2,$ have been generated the stabilization forms are given by:
\begin{align}\label{eq:stab-surf-glob}
s_{h,0}(v,w)&=\sum_{k=0}^1 \tau_{0,k} h^{2k-2}([\nabla^k v_0 ], [\nabla^k w_0])_{\mcF^*_{h,0}}
+\tau_{0,2} h^2(\nabla_{n_0} v_0 , \nabla_{n_0} w_0)_{\mcK_{h,0}}
\\ \label{eq:stab-bulk-glob}
s_{h,i}(v,w) &= \sum_{k=0}^1 \tau_{i,k} h^{2k-1} ([\nabla^k v_i ], [\nabla^k w_i])_{\mcF^*_{h,i}},  
\qquad i=1,2
\end{align}


\begin{figure}\centering
\includegraphics[width=0.32\textwidth]{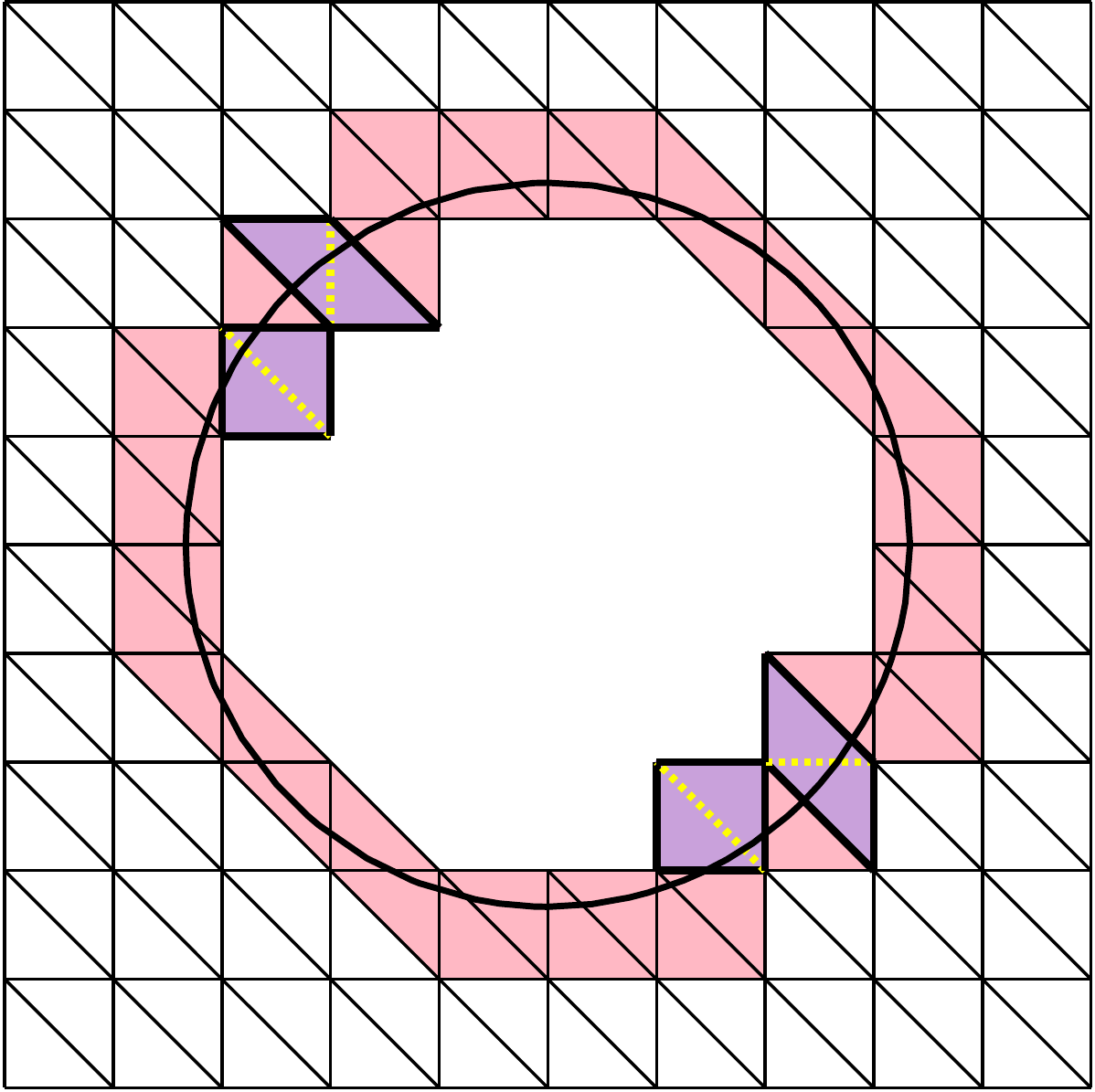} \hspace{0.3cm} 
\includegraphics[width=0.32\textwidth]{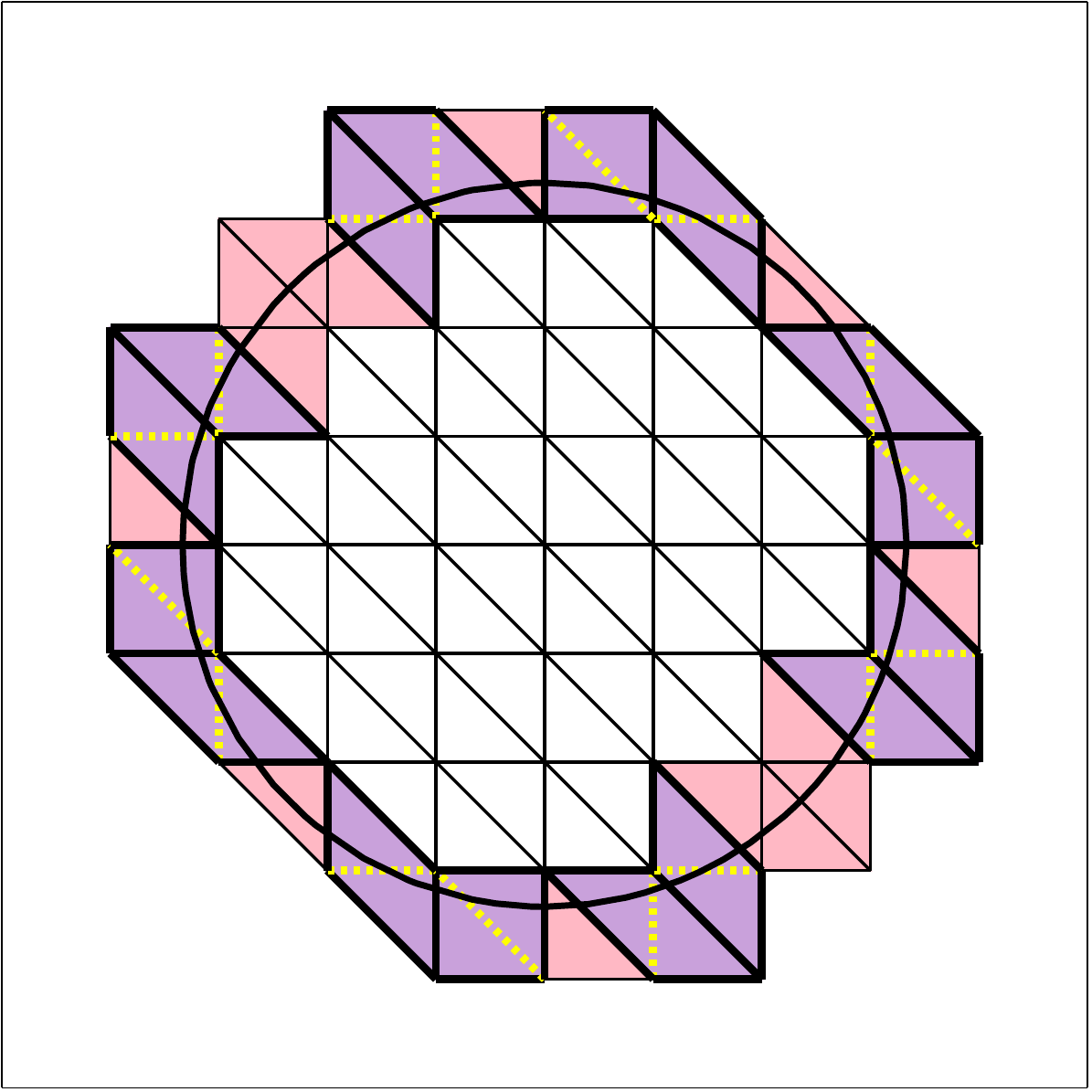} 
\caption{Illustration of macro elements and the patchwise stabilization. Elements in the active mesh $\mcT_{h,1}$ (left panel) and in $\mcT_{h,2}$ (right panel) are shown. Macro elements consisting of several triangles are marked in purple and their interior edges, the dotted lines, are edges on which stabilization is applied. There are in total 4 edges in $\mcF^*_{h,1}$ and 16 edges in $\mcF^*_{h,2}$. \label{fig:illuststab_bulk}}
\label{fig:a}
\end{figure} 

%
%
%
%

\subsection{Properties of the Macro Element Stabilization}

We now present some useful results on the properties provided by the macro element stabilization. A common theme is the observation that the macro element stabilization enables us to control the difference between a general discontinuous piecewise linear function on the active mesh and an approximation which is discontinuous piecewise linear on the macro elements. For the latter function standard finite element estimates holds since the macro elements have a large intersection with the domain.

The first bound shows that we can estimate the difference between a discontinuous piecewise linear function and a suitable linear function 
on a macro element by the stabilization terms on the macro element. The second bound is the stabilization property that follows essentially from \cite{LaZa19}, where continuous higher order element were considered and the macro elements appeared  as a technical tool in the proofs.

\begin{lem}\label{lem:macro-approx} There is a constant such that for all $v \in W_{h,i}$ and all macro elements $M\in \mcM_{h,i}$, $i=0,1,2,$
\begin{align}\label{eq:approx-est}
\inf_{v_M \in P_1(M)} \|  \nabla^m (v - v_M) \|_M^2 \lesssim  h^{2(1-m)}  \sum_{k=0}^1 h^{2k-1} \| [\nabla^k v ] \|^2_{\mcF_{h,i}(M)},\quad m=0,1
\end{align}
\end{lem}
\begin{proof} Let $T_M \in \mcT_{h,i}$ be the element in $M$ that has a large intersection with $\Omega_i$ and let $v_M \in P_1(M)$ be such that $v_M |_{T_M}= v|_{T_M}$. We recall that 
for two elements $T_1$ and $T_2$ sharing a face $F$ there is a constant such that for each 
$w = (w_1,w_2) \in P_1(T_1) \oplus P_1(T_2)$,
\begin{equation}\label{eq:stab-pairs}
\| \nabla^m w_1 \|^2_{T_1} \lesssim \| \nabla^m w_2 \|^2_{T_2} + h^{2(1-m)} \sum_{k=0}^1 h^{2k-1} \| [\nabla^k w ] \|^2_{F} 
\end{equation}
Setting $w = v - v_M$, noting that $v-v_M = 0$ on $T_M$, and then using  (\ref{eq:stab-pairs})  repeatedly together with the fact that the number of elements in each macro element is uniformly bounded completes the proof. 
\end{proof} 

\begin{lem} \label{lem:macro-control}
There is a constant such that for all $v \in W_{h,i}$ and all macro elements $M\in \mcM_{h,i}$, $i=0,1,2,$
\begin{align}\label{eq:stab-est}
&\| \nabla^m v \|_M^2 
\\
& \nonumber
\lesssim  
\begin{cases}
  \| \nabla_i^m v \|^2_{M \cap \Omega_i} 
+ h^{2(1-m)} \Big( \sum_{k=0}^1 h^{2k-1} \| [\nabla^k v ] \|^2_{\mcF_{h,i}(M)}  \Big), & i=1,2
\\
h \| \nabla_0^m v \|^2_{M \cap \Omega_0} 
+ h^{2(1-m)} \Big( \sum_{k=0}^1 h^{2k-1} \| [\nabla^k v ] \|^2_{\mcF_{h,0}(M)}  + h \| \nabla_n v \|^2_{M\cap \Omega_0}  \Big) &
\end{cases}
\end{align}
for $m=0,1.$
\end{lem}
\begin{proof} Let $T_M \in \mcT_{h,i}$ be the element in the macro element $M$ that has a large intersection 
with $\Omega_i$, see (\ref{eq:largeel}).  Let $v_M \in P_1(M)$ be such that
$v_M |_{T_M}= v|_{T_M}$. Then we have using (\ref{eq:approx-est}) in Lemma \ref{lem:macro-approx},
\begin{align}
\|\nabla^m v\|^2_M &\lesssim  \|\nabla^m v_M\|^2_M + \| \nabla^m ( v - v_M )\|^2_M
\\ \label{eq:stab-prf-aa}
&\lesssim  \|\nabla^m v_M\|^2_M + h^{2(1-m)} \sum_{k=0}^1 h^{2k-1} \| [\nabla^k w ] \|^2_{\mcF_{h,i}(M)} 
\end{align}
To estimate the first term on the right hand side in (\ref{eq:stab-prf-aa}) we have the following 
inverse inequality in the case of bulk domains
\begin{align} \label{eq:stab-prf-a}
\| \nabla^m v_M \|^2_M \lesssim \| \nabla^m v_M \|^2_{T_M} \lesssim  \| \nabla^m v_M \|^2_{T_M \cap \Omega_i}
\lesssim  \| \nabla^m v \|^2_{M \cap \Omega_i}, \quad i=1,2
\end{align}
For the surface domain, $i=0$, we instead of (\ref{eq:stab-prf-a}) have
\begin{align}
\| \nabla^m v_M \|^2_M &\lesssim \| \nabla^m v_M \|^2_{T_M}
\\
&\lesssim 
h\| \nabla^m_0 v_M \|^2_{T_M\cap \Omega_0} + h^{3 - 2m} \|\nabla_n v_M \|^2_{T_M\cap \Omega_0} 
\\ \label{eq:stab-prf-c}
&\lesssim h\| \nabla^m_0 v \|^2_{M\cap \Omega_0} + h^{2(1-m)} h \| \nabla_n v \|^2_{M\cap \Omega_0}  
\end{align}
where we recall that $\nabla_0$ is the tangential derivative associated with $\Omega_0$ and that the codimension for the surface 
is $d-d_0 = 1$.  We also note that  we need the normal stabilization, the last term on the right hand side of (\ref{eq:stab-surface}), 
at $\Omega_0$ to pass from $T_M$ to the intersection $T_M \cap \Omega_0$. Together, (\ref{eq:stab-prf-aa}),  (\ref{eq:stab-prf-a}), and  (\ref{eq:stab-prf-c}) 
complete the proof.
\end{proof}

We next show that, on the bulk domains, the macro element stabilization together with terms 
that are controlled by the weak form in the method, control face stabilization on all faces. Recall
that the stabilization is only added inside the macro elements and that the Nitsche penalty term only 
acts on the intersection between a face and the domain and is therefore significantly weaker compared 
to stabilization on all faces. To prepare for the estimate we first show a 
Poincar\'e type estimate for macro elements in Lemma \ref{lem:mavro-poincare}.

\begin{lem} \label{lem:mavro-poincare}  For the bulk domains ($i=1,2$), there 
is a constant such that for all $v \in W_{h,i}$ and all macro elements $M\in \mcM_{h,i}$,
\begin{align}\label{eq:stab-est-poincare}
\inf_{w_M \in P_0(M)} \|  v - w_M \|_M^2 \lesssim  h^2  \Big( \| \nabla v \|^2_{M \cap \Omega_i}  + 
 \sum_{k=0}^1 h^{2k-1} \| [\nabla^k v ] \|^2_{\mcF_{h,i}(M)} \Big) 
\end{align}
\end{lem}
\begin{proof} Let $T_M \in \mcT_{h,i}$ be the element in $M$ that has a large intersection with $\Omega_i$ and let $v_M \in P_1(M)$ 
be such that $v_M |_{T_M}= v|_{T_M}$, and $w_{M} \in P_0(M) $ be such that $w_{M}|_{T_M}$ is the $L^2$ projection of 
$v_M|_{T_M} = v|_{T_M}$ onto constants at $T_M$. Then we have 
\begin{align}
\| v - w_M \|^2_{M} 
&\lesssim 
\| v - v_M\|_M^2 + \| v_M - w_{M} \|^2_{T_M} 
\\
&
\lesssim  \| v - v_M\|_M^2 + h^2 \| \nabla v_M\|^2_{T_M}
\\ 
&
\lesssim  \| v - v_M\|_M^2 + h^{2} \| \nabla v_M\|^2_{T_M\cap \Omega_i}
\\ 
&
\lesssim  \| v - v_M\|_M^2 + h^{2}  \| \nabla (v -v_M) \|^2_{T_M\cap \Omega_i} +  h^{2} \| \nabla v \|^2_{T_M\cap \Omega_i}
\\ \label{eq:stab-prf-e} 
&
\lesssim h^2 \sum_{k=0}^1 h^{2k-1} \| [\nabla^k v ] \|^2_{\mcF_{h,i}(M)} 
+ h^2 \| \nabla v \|^2_{M\cap \Omega_i}  
\end{align}
where in inequality (\ref{eq:stab-prf-e}) we employed  Lemma \ref{lem:macro-approx}, with $m=0$ and $m=1$.
\end{proof}

\begin{lem}\label{lem:controlstab}  For the bulk domains ($i=1,2$), there 
is a constant such that for all $v \in W_{h,i}$, 
\begin{align}\label{eq:stab-eqv}
\sum_ {k=0}^1 h^{2k-1} \| [\nabla^k v] \|^2_{\mcF_{h,i}} \lesssim \| \nabla v \|^2_{\mcK_{h,i}} 
+ h^{-1} \| [ v ] \|^2_{\mcE_{h,i}} 
+\sum_{M \in \mcM_{h,i}} \sum_{k=0}^1 h^{2k-1} \| [\nabla^k v ] \|^2_{\mcF_{h,i}(M)}
\end{align}
\end{lem}

\begin{proof} Let us start with the estimation of $h^{-1}\| [ v ] \|^2_{\mcF_{h,i}}$. We note that the Nitsche penalty term provides control over $h^{-1}\| [ v ] \|^2_{\mcE_{h,i}}$ where we recall that $\mcE_{h,i} = \mcF_{h,i} \cap \Omega_i$,  and it therefore remains to control 
 $h^{-1}\| [ v ] \|^2_{F}$ for faces $F \in \mcF_{h,i}$ which intersect the boundary $\Omega_0$. If $F$ has a large intersection with $\Omega_i$, i.e. there is a constant such that
\begin{align}\label{eq:largeF}
 |F \cap \Omega_i | \gtrsim h^{d-1}
\end{align} 
and  an inverse inequality directly gives
 \begin{align}\label{eq:face-in}
 h^{-1} \| [ v ] \|^2_ F \lesssim h^{-1} \| [ v ] \|^2_ {F \cap \Omega_i}
 \end{align}
Next we consider the case when (\ref{eq:largeF}) does not hold. Let $F$ be such a face, which is shared by two macro elements $M_i$ and $M_j$. Then there 
 is a chain 
of macro elements $\{ M_n \}_{n=1}^l$ connecting $M_i$ to $M_j$ ($M_1 = M_i$ and $M_l = M_j$) 
of uniformly bounded length $l$, such that two consecutive macro elements $M_n$ and $M_{n+1}$ share a face $F_n$ such that $|F_n \cap  \Omega_i| \gtrsim h^{d-1}$. 
Observing that a macro element must have at least one face on its boundary that has a large intersection with $\Omega_i$, we can then use that face to pass to a neighboring macro element.
%

Starting with the estimation of $h^{-1} \| [ v ]\|^2_F$. Letting $v_n=v|_{M_n}$, for 
$n=1,\dots,l,$ and keeping in mind that $v_i = v_1$ and $v_j = v_l$, we have for any 
$w \in \oplus_{M \in \mcM_{h,i}} P_0(M)$, with $w_n = w|_{M_n} \in P_0(M)$ for $n=1,\dots,l$,
\begin{align}
h^{-1} \| [ v ] \|^2_F &\lesssim h^{-1} \| [ w ] \|^2_F 
+ h^{-1} \| v_1 - w_1 \|^2_F + 
 h^{-1} \| v_l - w_l \|^2_F
\\
&\lesssim h^{-1} \| [ w ] \|^2_F +  h^{-2} \| v_1 - w_1 \|^2_{M_1} + 
 h^{-2}\| v_l - w_l \|^2_{M_l}
\end{align}
Here we added and subtracted $w_1$ and $w_l$, used the triangle inequality, and an inverse 
 trace inequality on the macro elements.
To estimate the first term on the right hand side we add and subtract $w_2,\dots, w_{l-1}$ and 
use the triangle inequality 
\begin{align}
h^{-1} \| [ w ] \|^2_F & = h^{-1} \| w_1 - w_l \|^2_F 
\lesssim \sum_{n=1}^{l-1} h^{-1} \| w_n - w_{n+1} \|^2_F 
 \lesssim \sum_{n=1}^{l-1} h^{-1} \|  w_n - w_{n+1} \|^2_{F_n} 
\end{align}
where we at last used the fact that $ w_n -  w_{n+1} $ is constant and $|F| \sim |F_n | \sim h^{d-1}$ to pass from $F$ to $F_n$. Now we pass back to $v$ from $w$ by adding and subtracting $v_n$ and 
$v_{n+1}$, and then using the triangle inequality 
\begin{align}
h^{-1} \|[w] \|^2_{F_n} &\lesssim h^{-1} \|[v ]\|^2_{F_n} 
+ h^{-1} \|w_n - v_n \|^2_{F_n} + h^{-1} \| w_{n+1} - v_{n+1} \|^2_{F_n}
\\
&\lesssim h^{-1} \|[v ]\|^2_{F_n} 
+ h^{-2} \|w_n - v_n \|^2_{M_n} + h^{-2} \| w_{n+1} - v_{n+1} \|^2_{M_{n+1}}
\end{align}
Collecting the bounds, taking the infimum over $w\in \oplus_{M\in \mcM_{h,i}} P_1(M)$, and using the Poincar\'e estimate in Lemma \ref{lem:mavro-poincare} on the macro elements, give
\begin{align}
&h^{-1} \| [ v ] \|^2_F \lesssim \sum_{n=1}^{l-1} h^{-1} \|[ v ]\|^2_{F_n}  
+ \sum_{n=1}^{l} h^{-2} \|v_n - w_n \|^2_{M_n}
\\ \label{eq:face-out}
 &\quad \lesssim  \sum_{n=1}^{l-1} h^{-1} \|[v ]\|^2_{F_n \cap \Omega_i}
+  \sum_{n=1}^{l-1} \Big( \| \nabla v_n \|^2_{M \cap \Omega_i}  + 
 \sum_{k=0}^1 h^{2k-1} \| [\nabla^k v_n ] \|^2_{\mcF_{h,i}(M)} \Big) 
\end{align}
Finally, summing over all faces $F \in \mcF_{h,i}$ and using the bounds 
 (\ref{eq:face-in}) and (\ref{eq:face-out}) we obtain 
\begin{align}\label{eq:face-aa}
h^{-1} \| [ v ] \|^2_{\mcF_{h,i}} 
&\lesssim h^{-1} \|[v ]\|^2_{\mcF_{h,i}\cap \Omega_i}
+ \sum_{M \in \mcM_{h,i}} \| \nabla v \|^2_{M \cap \Omega_i}  + 
\sum_{k=0}^1 h^{2k-1} \| [\nabla^k v_n ] \|^2_{\mcF_{h,i}(M)} 
\end{align}

It remains to estimate $h \| [\nabla v] \|^2_F$ for any $F \in \mcF_{h,i} \setminus \mcF_{h,i}^*$, where we recall that $\mcF_{h,i}^*$ is the set of faces where stabilization is applied. Using an inverse trace inequality on the elements, collecting the elements into macro elements, and then using the stability estimate in Lemma \ref{lem:macro-control}, with $m=1$ and $i=1,2,$ we obtain
\begin{align}
h \| [ \nabla v ] \|^2_{\mcF_{h,i}} &\lesssim \| \nabla v \|^2_{\mcT_{h,i}}
= \sum_{M \in \mcM_{h,i}}\| \nabla v \|^2_M 
\\ \label{eq:face-bb}
&\qquad 
\lesssim \sum_{M \in \mcM_{h,i}} 
\Big(  \| \nabla^m v \|^2_{M \cap \Omega_i} 
+ \sum_{k=0}^1 h^{2k-1} \| [\nabla^k v ] \|^2_{F}   \Big) 
\end{align}
Summing (\ref{eq:face-aa}) and (\ref{eq:face-bb}) and using the fact that $\mcF_{h,i} \cap \Omega_i = \mcE_{h,i}$ and $\mcT_{h,i} \cap \Omega_i = \mcK_{h,i}$ completes the proof.
\end{proof} 
\begin{figure}
\begin{center}
\includegraphics[scale=0.6]{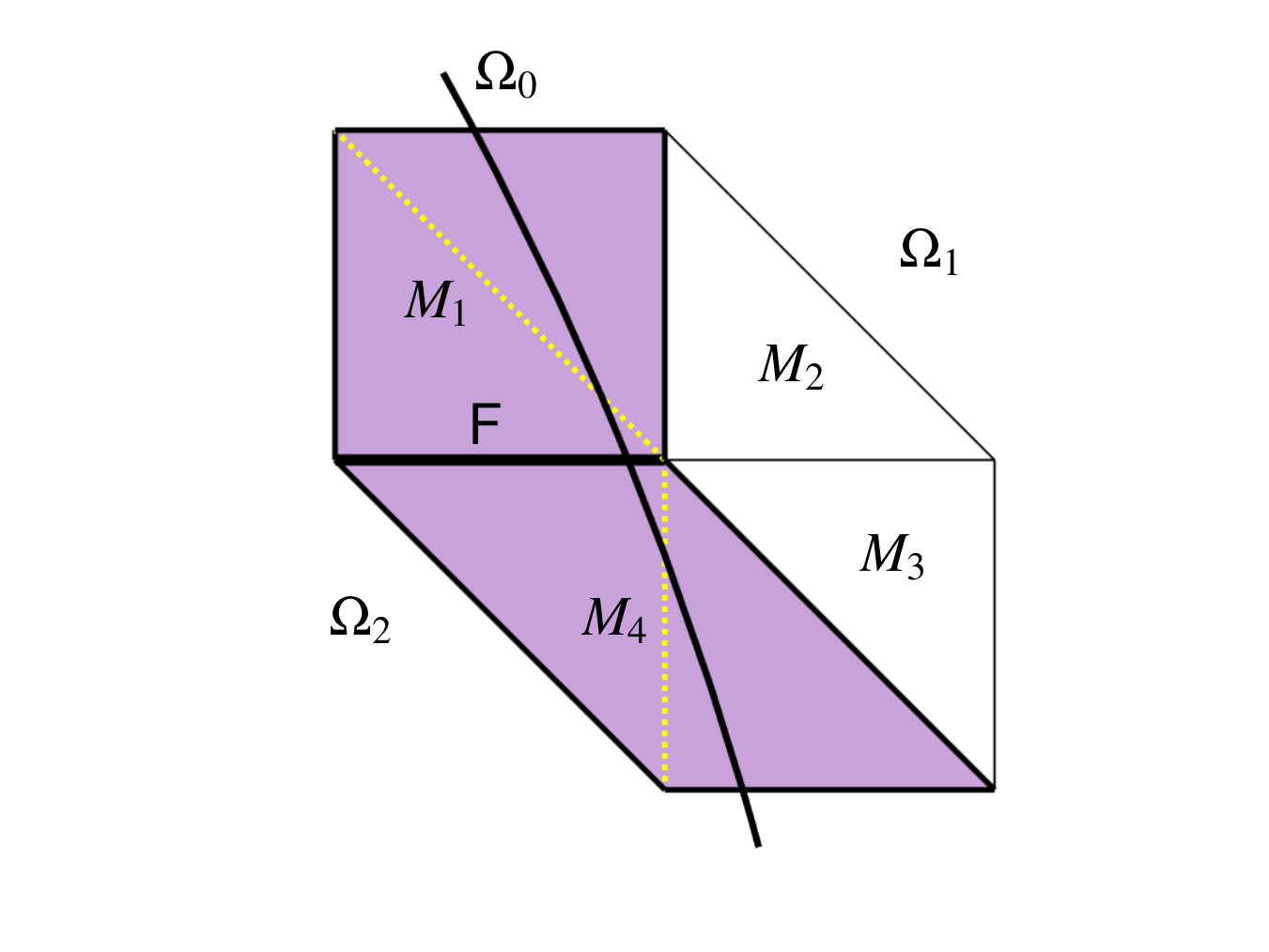}
\end{center}
\caption{The face $F$ with a small intersection with $\Omega_1$ (right side of the interface)  and a large intersection with $\Omega_2$ (left side of the interface) and 
the chain of macro elements $\{M_n\}_{n=1}^4$ that enables us to estimate the jump over the face
$F$.}
\end{figure}

We end this section by showing a trace inequality for the discontinuous function spaces on the bulk domains, 
which is of general interest, and will also be used in the derivation of a Poincar\'e inequality for the discrete 
spaces. The 
difficulty is that the functions are discontinuous so we can not directly employ a standard trace 
inequality and an element wise trace inequality would produce a non optimal factor of $h^{-1}$.

\begin{lem} \label{lem:discrete-trace} There is a constant such that for all $v \in W_{h,i}$, $i=1,2$,
\begin{align}\label{eq:trace-discont}
\| v \|^2_{\Omega_0} &\lesssim \| v \|^2_{H^1(\mcK_{h,i})} + h^{-1} \| [ v ] \|^2_{\mcF_{h,i}} 
\\
&
\lesssim  \| v \|^2_{H^1(\mcK_{h,i})} +  h^{-1} \| [ v ] \|^2_{\mcE_{h,i}} 
+\sum_{M \in \mcM_{h,i}} \sum_{k=0}^1 h^{2k-1} \| [\nabla^k v ] \|^2_{\mcF_{h,i}(M)}
\end{align}
\end{lem}
\begin{proof} Let us consider $i=1$ and let 
$O_{h,1}: W_{h,1} \rightarrow W_{h,1} \cap C(\Omega_{h,1})$ be the Oswald interpolant with nodal values 
given by the average of the nodal values of the discontinuous function $v$, more precisely for a node 
$x_i$ we define
\begin{equation}
O_{h,1} v|_{x_i} = \frac{1}{| \mcT_{h,1}(x_i)|} \sum_{T \in \mcT_{h,1}(x_i)} v|_T (x_i)  
\end{equation}
where $\mcT_h(x_i)$ is the set of elements in $\mcT_{h,1}$ sharing node $x_i$. We then have the well known estimate 
\begin{align}
\| v - O_{h,1} v \|^2_{\mcT_{h,1}} \lesssim h \|[v ] \|^2_{\mcF_{h,1}}
\end{align}
see \cite{BuEr07}. Adding and subtracting the Oswald interpolant and using the inverse trace inequality 
$\| w \|^2_{\partial \omega \cap T} \lesssim h^{-1}\| w \|^2_T$  for the difference 
  $  v - O_{h,1} v $ and the standard trace inequality  $\| w \|_{\Omega_0 } \lesssim \| w \|_{H^1(\Omega_1)}$, $w \in H^1(\Omega_1)$,  
  with $w = O_{h,1} v$  we get 
\begin{align}
\| v \|^2_{\Omega_0} &\lesssim \| v - O_{h,1} v \|^2_{\Omega_1}  
+ \| O_{h,1} v \|^2_{\Omega_1}  
\\
&\lesssim 
h^{-1} \| v - O_{h,1} v\|^2_{\mcT_{h,1}(\Omega_0)}  
+ \| O_{h,1} v  \|^2_{H^1(\Omega_1)}  
\\
&\lesssim 
h^{-1} \| v - O_{h,1} v\|^2_{\mcT_{h,1}(\Omega_0)}  
+ \| O_{h,1} v  - v \|^2_{H^1(\mcT_{h,1})}  
+ \| v \|^2_{H^1(\mcT_{h,1}\cap \Omega_1)}  
\\
&\lesssim  (1 + h^{-1} + h^{-2} )\| v - O_{h,1} v \|^2_{\mcT_{h,1}}  + \|  v \|^2_{H^1(\mcK_{h,1})}  
\\
&\lesssim h^{-1} \| [ v ] \|^2_{\mcF_{h,1}} + \|  v \|^2_{H^1(\mcK_{h,i})}  
\\
&\lesssim \| v \|^2_{H^1(\mcK_{h,i})} +  h^{-1} \| [ v ] \|^2_{\mcE_{h,i}} 
+\sum_{M \in \mcM_{h,i}} \sum_{k=0}^1 h^{2k-1} \| [\nabla^k v ] \|^2_{\mcF_{h,i}(M)}
\end{align}
for all $h \in (0,h_0]$, and we used Lemma \ref{lem:controlstab} in the last inequality.
\end{proof}

\subsection{Derivation of the Method and Consistency}
\label{sec:derivation}
To derive the method we test the exact bulk equations (\ref{eq:BVP}) by $\tkapi v_i$ with $v_i \in W_{h,i}$, $i=1,2,$ and 
integrate element wise using Green's formula. Using the assumptions on the velocity field $\beta$ we obtain
\begin{align}
&\sum_{i=1}^2 (f_i,\tkapi v_i)_{\Omega_i} 
=
\sum_{i=1}^2 -(\nabla \cdot (A_i \nabla u_i), \tkapi v_i)_{\underbrace{\mcT_{h,i} \cap \Omega_i}_{=\mcK_{h,i}}}
+(\nabla \cdot (\beta u_i),\tkapi v_i)_{\underbrace{\mcT_{h,i} \cap \Omega_i}_{=\mcK_{h,i}}}
\\
&
=
\sum_{i=1}^2\Big( ( A_i \nabla u_i,\nabla   \tkapi  v_i)_{\mcK_{h,i}}-(\nu_i \cdot A_i \nabla u_i,    \tkapi  v_i)_{\partial \mcK_{h,i}} \Big)
\\
&\qquad 
+\frac{1}{2} \Big( (\beta \cdot \nabla u_i, \tkapi v_i)_{\mcK_{h,i}} - (u_i,\beta\cdot \nabla  ( \tkapi  v_i ) )_{\mcK_{h,i}} + ((\nu_i \cdot \beta) u_i, \tkapi v_i)_{\partial \mcK_{h,i}} \Big) 
\\
&=
\sum_{i=1}^2 \tkapi \Big( ( A_i \nabla u_i,\nabla v_i)_{\mcK_{h,i}}  - (\langle \nu_i \cdot A_i \nabla u_i \rangle ,[v_i])_{\mcE_{h,i}} \Big)
\\
&\qquad 
 -
 \sum_{i=1}^2  \underbrace{(n_i \cdot A_i \nabla u_i, \tkapi v_i - v_0)_{\partial \Omega_i \cap \Omega_0}}_{=-\kappa_{0,i}^{-1}([\kappa u]_i, [\kappa v]_i)_{\partial \Omega_i \cap \Omega_0}}
- \underbrace{\sum_{i=1}^2 (n_i \cdot A_i \nabla u_i, v_0)_{= \partial \Omega_i \cap \Omega_0}}_{=(\llbracket n\cdot A \nabla u \rrbracket, v_0)_{\Omega_0} }
\\
&\qquad 
+
\sum_{i=1}^2 \tkapi \frac{1}{2}\Big( (\beta \cdot \nabla u_i,v_i)_{\mcK_{h,i}} -(u_i,\beta \cdot \nabla v_i)_{\mcK_{h,i}} +((\nu_i \cdot \beta)\langle u_i \rangle,[v_i])_{\mcE_{h,i}} \Big)
\end{align}
Since the components $v_i$ of $v$ are discontinuous across the internal faces in $\mcT_{h,i}$, the partial integration manufactures certain jump terms. By adding terms 
that are zero we can make the diffusion dependent parts symmetric and the convection dependent parts skew-symmetric. 
\begin{align}
&\sum_{i=1}^2 (f_i,\tkapi v_i)_{\Omega_i} 
 =\sum_{i=1}^2 \tkapi\Big( ( A_i \nabla u_i, \nabla v_i)_{\mcK_{h,i}}  
 -(\langle \nu_i \cdot A_i \nabla u_i \rangle ,[v_i])_{\mcE_{h,i}}  
 \\
&\qquad \qquad \qquad \qquad \qquad
 - \underbrace{([u_i],\langle \nu_i \cdot A_i \nabla v_i \rangle )_{\mcE_{h,i}}}_{={0}} 
 + \underbrace{(\lambda_{a_i} h^{-1} [u_i], [v_i])_{\mcE_{h,i}}}_{={0}} 
 \Big)
 \\
&\qquad 
+\sum_{i=1}^2\tkapi \frac{1}{2}\Big( (\beta \cdot \nabla u_i,v_i)_{\mcK_{h,i}} -(u_i,\beta \cdot \nabla v_i)_{\mcK_{h,i}} 
\\
&\qquad \qquad \qquad
+ ((\nu_i \cdot \beta)\langle u_i \rangle  ,[v_i])_{\mcE_{h,i}}  - \underbrace{((\nu_i \cdot \beta)[u_i],\langle v_i \rangle )_{\mcE_{h,i}}}_{={0}}\Big)+ \tkapi \underbrace{(\lambda_{b_i} [u_i], [v_i])_{\mcE_{h,i}}}_{={0}} 
\\
&\qquad
+\sum_{i=1}^2 ( \kappa_{0,i}^{-1} [\kappa u]_i,[\kappa v]_i)_{\Omega_0}
- (\llbracket n \cdot A \nabla v\rrbracket,v_0)_{\Omega_0}
\\
&=
\sum_{i=1}^2 \tkapi a_{h,i}(u_{i},v_i)  + \tkapi b_{h,i}(u_{i},v_i)  
+  (\kappa_{0,i}^{-1} [\kappa u]_i,[\kappa v]_i)_{\Omega_0}
\underbrace{-(\llbracket n \cdot A \nabla v\rrbracket,v_0)_{\Omega_0}}_{\bigstar}
\end{align}
Next using partial integration on the surface $\Omega_0$ we obtain
\begin{align}
\bigstar &= -(f_0,v_0)_{\Omega_0} - (\nabla_0 (\cdot A_0 \nabla_0 u_0), v_0)_{\Omega_0}
+ (\nabla_0 \cdot (\beta u_0),v_0)_{\Omega_0} 
\\
&= -(f_0,v_0)_{\Omega_0}  + (A_0 \nabla_0 u_0, \nabla_0 v_0)_{\mcK_{h,0}} 
- (\nu_0 \cdot A_0 \nabla_0 u_0, v_0)_{\partial \mcK_{h,0}}  
\\
&\qquad 
+\frac{1}{2}\left(( \beta \cdot \nabla_0 u_0,v_0)_{\mcK_{h,0}} - (u_0,\beta \cdot \nabla_0 v_0)_{\mcK_{h,0}} +((\nu_0 \cdot \beta) u_0, v_0)_{\partial \mcK_{h,0}}\right)
\\
&=-(f_0,v_0)_{\Omega_0}  + (A_0 \nabla_0 u_0, \nabla_0 v_0)_{\mcK_{h,0}} - (\langle \nu_0 \cdot A_0 \nabla_0 u_0 \rangle, [v_0])_{\mcE_{h,0}}
\\
&\qquad 
   - \underbrace{([u_0], \langle \nu_0 \cdot A_0 \nabla_0 v_0 \rangle )_{\mcE_{h,0}}}_{=0} 
    +\underbrace{(\lambda_{a_0}h^{-1} [u_0] ,  [v_0]  )_{\mcE_{h,0}}}_{=0} 
\\
&\qquad 
+\frac{1}{2}\Big(( \beta \cdot \nabla_0 u_0,v_0)_{\mcK_{h,0}} -( u_0,\beta \cdot \nabla_0 v_0)_{\mcK_{h,0}} \Big)
\\
&\qquad 
+
\frac{1}{2}\Big( ((\nu_0 \cdot \beta) \langle u_0 \rangle, [v_0])_{\mcE_{h,0}} 
- \underbrace{((\nu_0 \cdot \beta)[u_0], \langle v_0 \rangle )_{\mcE_{h,0}}}_{=0} 
\Big)
  +  \underbrace{(\lambda_{b_0}[u_0] ,  [v_0]  )_{\mcE_{h,0}}}_{=0} 
   \\
   &= - (f_0,v_0)_{\Omega_0} +  a_{h,0}(u_0,v_0)+ b_{h,0}(u_0,v_0)
\end{align}
We finally add the stabilization terms defined by (\ref{eq:stab-surf-glob}) and (\ref{eq:stab-bulk-glob}) to 
the formulation. The above derivation, together with the fact that the stabilization terms vanish on the exact solution, prove the next lemma, which states that the proposed discontinuous cut finite element method is consistent. 
\begin{lem}
The solution $u\in W \cap \bigoplus_{i=0}^2 H^2(\Omega_i)$ to the convection diffusion problem \eqref{eq:BVP}-\eqref{eq:BVP-d} satisfies
\begin{equation}\label{eq:consistency}
  A_h(u,v)=L_h(v) \qquad \forall v \in W_h
\end{equation}
\end{lem}

\subsection{The Local Conservation Property}

We shall now establish a conservation property for the so called Nitsche flux on the macro elements. Since we do not have any additional 
stabilization at the interfaces between the macro elements the natural local conservation property inherent in the discontinuous Galerkin 
formulation is preserved on the macro element level. Consider a macro element $M \in \mcM_{h,i}$. Testing with the characteristic function 
$\chi_M$ associated with the macro element $M$, yields  
\begin{align}
L_h(\chi_M) = A_h(u_h,\chi_M)
\end{align}
where 
\begin{align}
L_h(\chi_M) = (\tkapi f_i,1)_{M \cap \Omega_i }
\end{align}
and using the fact that $s_{h,i}(u_h,\chi_M) = 0$ since there is no stabilization across the macro element boundaries we get 
\begin{align}
  A_h(u_h,\chi_M) &= \sum_{i=0}^2 a_{h,i}(u_{h,i},\chi_M) + b_{h,i}(u_{h,i},\chi_M) + \sum_{i=1}^2 ([\kappa u_h]_i, [\chi_M]_i)_{M \cap \Omega_0}
\end{align}
with 
\begin{align}
a_{h,i}(u_{h,i},\chi_M)= -(\langle \nu_i \cdot A_i \nabla_i u_{h,i} \rangle ,\chi_M )_{\partial M \cap \mcE_{h,i}} 
+ (\lambda_{a_i} h^{-1} [u_{h,i}], \chi_M)_{\partial M \cap \mcE_{h,i} }
\end{align}
and 
\begin{align}
& b_{h,i}(u_{h,i},\chi_M)= \frac{1}{2}
 \left( (\beta_i \cdot \nabla_i u_{h,i},\chi_M)_{M \cap \mcK_{h,i}}
 -((\nu_i \cdot \beta_i)[u_{h,i}],\langle \chi_M \rangle)_{\overline{M} \cap \mcE_{h,i}} \right)
\\
&\qquad \qquad \qquad +\frac{1}{2}((\nu_i \cdot \beta_i)\langle u_{h,i} \rangle  ,\chi_M )_{\partial M \cap \mcE_{h,i}} 
+(\lambda_{b_i} [u_{h,i}], \chi_M )_{\partial M \cap \mcE_{h,i}}
\\
&= ((\nu_i \cdot \beta_i)\langle u_{h,i} \rangle  , \chi_M )_{\partial M \cap \mcE_{h,i}} - \frac12 ((\text{div}_i \beta_i) u_{h,i},\chi_M)_{M \cap \mcK_{h,i}}
+(\lambda_{b_i} [u_{h,i}], \chi_M )_{\partial M \cap \mcE_{h,i}}
\end{align}
where we used partial integration and relation \eqref{eq:jumpaverrel}. Introducing the discrete normal flux $\Sigma_n(u_{h,i})$ such that 
\begin{align}
\Sigma_n (u_{h,i}) =  \langle \nu_i \cdot A_i \nabla_i u_{h,i} \rangle - ( h^{-1} \lambda_{a_i} + \lambda_{b_i} )  [u_{h,i}]  - (\nu_i \cdot \beta_i)\langle u_{h,i} \rangle, 
\qquad \text{on $\partial M \cap \mcE_{h,i}$}
\end{align}
We get, for each macro element $M$ in one of the bulk macro meshes $\mcM_{h,i}$, the conservation law
\begin{align}
( \Sigma_n (u_{h,i} ), 1)_{\partial M \cap \Omega_i} 
+ ([\kappa u_h]_i,1)_{M \cap \Omega_0} 
+ \frac12 ((\text{div}_i \beta_i) u_{h,i},1)_{M \cap \Omega_i} + (f_i,1)_{M \cap \Omega_i } 
= 0
\end{align}
and for $M$ a macro element in the surface mesh $\mcM_{h,0}$ we get the same expression except for the sign of the second term, which couples the bulk domains to the interface 
\begin{align}
&( \Sigma_n (u_{h,0} ), 1)_{\partial M \cap \Omega_0} 
- \sum_{i=1}^2 ([\kappa u_h]_i,1)_{M \cap \Omega_0} 
\\
&\qquad 
+ \frac12 ((\text{div}_0 \beta_0) u_{h,0},1)_{M \cap \Omega_i} + (f_0,1)_{M \cap \Omega_i } 
= 0
\end{align}

\section{Analysis of the Method}

\subsection{Properties of the Forms}

We establish the basic properties of the form $A_h$ including a discrete Poincar\'e inequality, coercivity, and continuity.   In order to identify the proper local scaling of the different terms in the method we will take the variable coefficients $A_i$ into account in the proof of coercivity.  However, to keep the overall presentation as simple as possible we elsewhere work with the global bounds on parameters $A_i$. We start by defining the norms
\begin{align} \label{eq:dg-norm}
\tn v_i \tn_{h,i}^2 
&=\| \nabla_i v_i \|^2_{A_i, \mcK_{h,i}} + h\| \langle \nabla_i v_i \rangle \|^2_{A_i, \mcE_{h,i}}  + h^{-1}\|\nu_i [v_i]\|^2_{A_i,\mcE_{h,i}} + \| v_i \|^2_{s_{h,i}}
\end{align}
and 
\begin{align}
\tn v_i \tn_{h,i,\bigstar}^2 
&=  \tn v_i \tn_{h,i}^2  + \| v_i \|^2_{\mcK_{h,i}} + h\| \langle v_i \rangle \|^2_{\mcE_{h,i}}
\end{align}
for $v_i\in H^s(\mcT_{h,i}) +  W_{h,i}$, $s>3/2$, and
\begin{align} \label{eq:dg-norm-sum}
\tn v \tn_{h}^2&=\sum_{i=0}^2 \tn v_i \tn_{h,i}^2 + \sum_{i=1}^2 \kappa_{0,i}^{-1} \| [\kappa v ]_i \|^2_{\Omega_0}
\\  \label{eq:dg-norm-star}
\tn v \tn_{h,\bigstar}^2&=\sum_{i=0}^2 \tn v_i \tn_{h,i,\bigstar}^2 + \sum_{i=1}^2 \kappa_{0,i}^{-1} \| [\kappa v ]_i \|^2_{\Omega_0}
\end{align}
Here $\| w \|^2_{A,\omega} = ( A w, w)_\omega$ is the $A$ weighted $L^2$ norm of $w$ over $\omega$. We now turn to the discrete Poincar\'e estimate.


\begin{lem}\label{lem:poincare-discrete} There is a constant such that for all $v \in W_h$, 
\begin{equation}\label{eq:poincare-discrete}
\sum_{i=0}^2 h^{-(d - d_i)} \| v_i \|^2_{\mcT_{h,i}} \lesssim \tn v \tn_h^2
\end{equation}
\end{lem} 
\begin{proof} The proof is similar to the proof of the continuous Poincar\'e inequality (\ref{eq:poincare-cont}) but a bit more complicated 
since we work with discontinuous piecewise polynomials. First we use Lemma 
\ref{lem:macro-control}, with $m=0$, and the fact $h \leq h_0 \lesssim 1$ to conclude 
that 
\begin{align}
\|v_i \|_M^2 &\lesssim 
  \| v_i\|^2_{M \cap \Omega_i} 
+ h^{2} \Big( \sum_{k=0}^1 h^{2k-1} \| [\nabla^k v_i ] \|^2_{\mcF_{h,i}(M)}  \Big)
\\
&
\lesssim   \| v_i \|^2_{M \cap \Omega_i}  + \| v_i \|^2_{s_{h,i,M}} , \qquad i=1,2
\end{align}
for the bulk domains $\Omega_i$, $i=1,2,$ and for the surface domain $\Omega_0$,
\begin{align}
h^{-1} \|v_0 \|_M^2 &\lesssim 
\| v_0 \|^2_{M \cap \Omega_0} 
+ h^{2} \Big( \sum_{k=0}^1 h^{2k-2} \| [\nabla^k v ] \|^2_{\mcF_{h,0}(M)} \Big) + h^2 \| \nabla_n v \|^2_{M\cap \Omega_0}  
\\
&\lesssim \| v_0 \|^2_{M \cap \Omega_0}  + \| v \|^2_{s_{h,0,M}}
\end{align}
We conclude that  
\begin{align}
\sum_{i=0}^2 h^{-(d-d_i)} \| v_i \|^2_{\mcT_{h,i}} \lesssim \sum_{i=0}^2 \| v_i \|^2_{\mcK_{h,i}} + \| v_i \|^2_{s_{h,i}}
\lesssim \underbrace{\| v_0 \|^2_{\mcK_{h,0}} }_{I} + \underbrace{\sum_{i=1}^2 \| v_i \|^2_{\mcK_{h,i}}}_{II}  + \tn v \tn^2_h
\end{align}
We now continue with estimates of terms $I$ and $II$.

\paragraph{Term $\bfI$.} Adding and subtracting suitable terms and using the 
discrete trace inequality in Lemma \ref{lem:discrete-trace}, with $i=2$, we get
\begin{align}
 \|v_0\|^2_{\Omega_0} &\lesssim \kappa_{2,0}^2 \|v_0\|^2_{\Omega_0}
\\
&\lesssim  \| \kappa_{2,0} v_0 - \kappa_2 v_2 \|^2_{\Omega_0} + 
\kappa_2^2 \|v_2 \|^2_{\Omega_0} 
\\
&\lesssim \| \kappa_{2,0} v_0 - \kappa_2 v_2 \|^2_{\Omega_0} + \| v_2 \|^2_{H^1(\mcK_{h,2})} 
+ \| v_2 \|^2_{s_{h,2}} +  h^{-1} \| [ v ] \|^2_{\mcE_{h,i}} 
\\ \label{eq:poin-b-I}
&\lesssim \underbrace{\| v_2 \|^2_{\Omega_2}}_{\leq II} + \tn v \tn^2_h
\end{align}
where we used the fact that $\kappa_{0,2}$ and $\kappa_2$ are positive constants. 
\paragraph{Term $\bfI\bfI$.} Let $\phi$ be the solution to the problem 
\begin{equation}
-\Delta \phi = \psi\quad \text{in $\Omega$}, \qquad \phi = 0 \quad \text{on $\partial \Omega$}
\end{equation}
where $\psi \in L^2(\Omega)$. (Note that this is not an interface problem.) Multiplying $-\Delta \phi = \psi$ by $v = \sum_{i=1}^2 \tkapi v_i \chi_i$, where $\chi_i$ is the characteristic function of $\Omega_i$, and integrating by parts we obtain
\begin{align}
&\sum_{i=1}^2 (\tkapi v_i,\psi)_{\Omega_i} = \sum_{i=1}^2 - (\tkapi v_i, \Delta \phi )_{\Omega_i} 
\\
&=  \sum_{i=1}^2 (\tkapi \nabla v_i,\nabla \phi )_{\mcK_{h,i}}  - (\tkapi v_i, \nabla_{\nu_i} \phi )_{\partial \mcK_{h,i}} 
\\
&=  \sum_{i=1}^2 (\tkapi \nabla v_i,\nabla \phi )_{\mcK_{h,i}} - (\tkapi [v], \nabla_{\nu_i} \phi )_{\mcE_{h,i}} 
- (\tkapi v_i - v_0, \nabla_{n_i} \phi )_{\Omega_0}  - ( v_0, \llbracket \nabla_{n_i} \phi \rrbracket )_{\Omega_0} 
\\
&=  \sum_{i=1}^2 (\tkapi \nabla v_i,\nabla \phi )_{\Omega_i}   - (\tkapi [v], \nabla_{\nu_i} \phi )_{\mcE_{h,i}}  
- (\kappa_{0,i}^{-1}[\kappa v]_i, \nabla_{n_i} \phi )_{\Omega_0}  
\\
&\leq  \Big(\sum_{i=1}^2 \tkapi \| \nabla v_i\|^2_{\Omega_i}  + \tkapi h^{-1} \|[v]\|^2_{\mcE_{h,i}}   + \kappa_{0,i}^{-1} \|[\kappa v]_i\|^2_{\Omega_0} \Big)^{1/2} 
\\
&\qquad \times 
\underbrace{\Big( \sum_{i=1}^2 \tkapi \|\nabla \phi \|^2_{\Omega_i}   + \tkapi h \|\nabla_{\nu_i} \phi \|^2_{\mcE_{h,i}}   + \kappa_{0,i}^{-1}\|\nabla_{n_i} \phi \|^2_{\Omega_0}  \Big)^{1/2}}_{\bigstar \lesssim \| \psi \|_\Omega}
\\ \
&\lesssim \tn v \tn_h \| \psi \|_\Omega
\end{align}
where we used the uniform coercivity of $A_i$ to pass to the $A_i$ weighted energy norm.
To establish the bound $\bigstar \lesssim \| \psi \|_\Omega$ we use a trace inequality on $\Omega_2$ followed by 
elliptic regularity, which holds since $\Omega$ is a convex polygonal domain, to conclude that 
\begin{equation}
 \|\nabla_{n_i} \phi \|_{\Omega_0} 
\leq \|\nabla \phi \|_{\Omega_0} 
\lesssim
 \| \phi \|_{H^2(\Omega_2)}
 \lesssim
 \| \phi \|_{H^2(\Omega)}
 \lesssim \| \psi \|_\Omega
\end{equation}
which holds for $i=1,2,$ since the gradient of $\phi \in H^1(\Omega)$. Then we may 
apply element wise trace inequalities followed by elliptic regularity to estimate the 
contributions from the faces 
\begin{align}
&h \|\nabla_{\nu_i} \phi \|^2_{\mcE_{h,i}}  
\lesssim \sum_{T \in \mcT_{h,i}}  h \|\nabla  \phi \|^2_{\partial T}
  \lesssim \sum_{T \in \mcT_{h,i}}  \|\nabla \phi \|^2_{T} + h^2 \|\nabla^2 \phi \|^2_{T} 
\\
&\qquad  \qquad  
  \lesssim \| \phi \|^2_{H^2(\mcT_{h,i})}
  \lesssim \| \phi \|^2_{H^2(\Omega)}
  \lesssim \| \psi \|^2_\Omega
\end{align}
and finally using the energy stability $\| \nabla \phi \|_{\Omega_i} \lesssim \| \psi \|_\Omega$ to 
conclude the estimate of $\bigstar$. We thus arrive at the estimate 
\begin{equation}
\sum_{i=1}^2 (\tkapi v_i,\psi)_{\Omega_i} \lesssim  \tn v \tn_h \| \psi \|_\Omega
\end{equation}
Setting $\psi_i = v_i$ and using the fact that $\tkapi>0$ we get 
\begin{equation}\label{eq:poin-a-II}
II = \sum_{i=1}^2 \|v_i\|^2_{\Omega_i} \lesssim  \tn v \tn^2_h 
\end{equation}

Together the bounds (\ref{eq:poin-b-I}) and (\ref{eq:poin-a-II}) of $I$ and $II$ complete the proof.
\end{proof}

\begin{lem} \label{lem:poincare-discrete-simple} There is a constant such that for all $v \in W_h$, 
\begin{equation}\label{eq:poincare-discrete-simple}
\tn v \tn_{h,\bigstar} \lesssim \tn v \tn_h
\end{equation}
\end{lem}
\begin{proof} Recalling the definition  of the norms (\ref{eq:dg-norm})--(\ref{eq:dg-norm-star}),
using the inverse estimates
\begin{gather}
\| v_i \|^2_{\mcK_h} \lesssim h^{-(d - d_i)} \| v_i \|^2_{\mcT_{h,i}}
\\
h\| \langle v_i \rangle \|^2_{\mcE_{h,i}} 
\lesssim \sum_{T \in \mcT_{h,i}} h \| v_i \|^2_{\partial T \cap \Omega_i} 
\lesssim \sum_{T \in \mcT_{h,i}}  h^{-(d - d_i) } \| v_i \|^2_{T}  = h^{-(d - d_i) } \| v_i \|^2_{\mcT_{h,i}}
\end{gather}
followed by the Poincar\'e estimate in Lemma \ref{lem:poincare-discrete}, we directly obtain the desired estimate
\begin{align}
\tn v \tn_{h,\bigstar}^2 
&=  \tn v \tn_{h}^2  + \sum_{i=0}^2 \| v_i \|^2_{\mcK_{h,i}} + h\| \langle v_i \rangle \|^2_{\mcE_{h,i}}
\\
&\qquad \lesssim \tn v \tn_{h}^2  + \sum_{i=0}^2 h^{-(d-d_i)}\| v_i \|^2_{\mcT_{h,i}} 
\lesssim  \tn v \tn_{h}^2
\end{align}
\end{proof}

To show coercivity of the form $A_h$ we will start by proving that the forms $a_{h,i} + s_{h,i}$ 
are coercive. Here, as we mentioned above, we will take the variable coefficients $A_i$ into account in order to identify the proper scalings of the Nitsche penalty and stabilization terms.

\begin{lem} \label{lem:macro-var}
There is a constant such that for all $v \in W_{h,i}$ and all macro elements $M\in \mcM_{h,i}$, $i=0,1,2,$
\begin{align}\label{eq:stab-est-var}
h \| \nabla_i v \|_{A_i,\mcE_{h,i}}^2 
\lesssim  
c_{A_i,M}  \| \nabla_i v \|^2_{A_i,M \cap \Omega_i} +  \| v \|^2_{s_{h,i,M}}
\end{align}
where the hidden constant is independent of $A_i$ and $ c_{A_i,M} = \| A_i \|_{L^\infty(M)} / 
\inf_{x\in M} \alpha_i(x)$ where $\alpha_i(x)$ is the coercivity constant of $A_i(x)$, see (\ref{eq:coer-Ai}). 
\end{lem}
\begin{proof} For $M \in \mcM_{h,i}$ let $\mcF_{h,i}(\overline{M}) = \{ F \in \mcF_{h,i} : F \subset \overline{M}\}$ be the faces 
in $M$ including the faces on the boundary of $M$ and let $\mcE_{h,i}(\overline{M}) = \mcF_{h,i}(\overline{M}) \cap \Omega_i$ be 
the intersection of the faces with $\Omega_i$. Then  $\mcE_{h,i} = \cup_{M \in \mcM_{h,i} }  \mcE_{h,i}(\overline{M})$ 
and we have
\begin{align}
h \| \nabla_i v \|_{A_i,\mcE_{h,i}}^2 
&\lesssim \sum_{M\in \mcM_{h,i}} h \|A_i \|_{L^\infty(M)} \| \nabla_i v \|^2_{\mcE_{h,i}(\overline{M})}
\\ \label{eq:macro-var-a}
& \lesssim \sum_{M\in \mcM_{h,i}}  \|A_i \|_{L^\infty(M)} \Big( \underbrace{h \| \nabla_i (v - v_M) \|^2_{\mcE_{h,i}(\overline{M})}}_{=I} 
 + 
\underbrace{h \| \nabla_i  v_M \|^2_{\mcE_{h,i}(\overline{M})}}_{=II}\Big)
\end{align}
where we added and subtracted $v_M \in P_1(M)$, such that $v_M|_{T_M} = v|_{T_M}$ where 
$T_M$ is the element with a large intersection with $\Omega_i$.

\paragraph{Term $\bfI$.} We use an inverse inequality to pass from the $d_i-1$ dimensional intersections 
$E = F \cap \Omega_i$ to the $d$ dimensional element $T$, manufacturing a scaling with $h^{- ( d - ( d_i - 1) )}$,
and then we apply Lemma \ref{lem:macro-approx},
\begin{align}
h  \| \nabla_i (v - v_M) \|^2_{\mcE_{h,i}(\overline{M})} 
&\lesssim 
h h^{- ( d - (d_i - 1)) } \| \nabla (v - v_M) \|^2_M
\\ \label{eq:macro-vara-I}
&\lesssim h^{-(d - d_i)}  \Big( \sum_{k=0}^1 h^{2k- 1} \| [ \nabla^k v  ] \|^2_{\mcF_{h,i}(M)} \Big)
\end{align}

%

\paragraph{Term $\bfI\bfI$.}  For the bulk domains ($i=1,2$) we use an inverse estimate 
to pass from the $d-1$ dimensional intersection $E = F \cap \Omega_i$ to the $d$ dimensional element $T$, then 
again using an inverse estimate we pass to the element with a large intersection $T_M$, then we add and subtract $v$,  use the triangle inequality, and finally Lemma \ref{lem:macro-approx},  
\begin{align}
&h \| \nabla_i  v_M \|^2_{\mcE_{h,i}(\overline{M})} \lesssim  \| \nabla_i  v_M \|^2_M \lesssim \| \nabla_i  v_M \|^2_{T_M}
\\
&\qquad \lesssim \| \nabla_i  (v_M - v) \|^2_{T_M} + \| \nabla_i  v \|^2_{T_M} 
\lesssim \sum_{k=0}^1 h^{2k- 1} \| [ \nabla^k v  ] \|^2_{\mcF_{h,i}(M)}  + \underbrace{\| \nabla_i  v \|^2_{T_M \cap \Omega_i}}_{\lesssim \| \nabla_i  v \|^2_{M \cap \Omega_i}}
\end{align}
We then have
\begin{align}
h \| \nabla_i v \|_{A_i,\mcE_{h,i}}^2 &\lesssim 
\sum_{M\in \mcM_{h,i}} \| A_i \|_{L^\infty(M)} \| \nabla_i v \|^2_{M \cap \Omega_i}  
\\
&\qquad 
+  
\underbrace{\| A_i \|_{L^\infty(M)} \Big( \sum_{k=0}^1 h^{2k - 1} \|[ \nabla^k v ] \|^2_{\mcF_{h,i}(M)}\Big) }_{\sim \| v \|^2_{s_{h,i}}}
\\
&\lesssim 
\sum_{M\in \mcM_{h,i}} \underbrace{\| A_i \|_{L^\infty(M)} 
 \alpha_{i,M}^{-1}}_{c_{A_i,M}} \| \nabla_i v \|^2_{A_i, M\cap \Omega_i}  
 +  \| v \|^2_{s_{h,i,M}}
\end{align}

For the surface domain ($i=0$) we need a more refined argument to find the required scaling on the normal gradient stabilization. To that 
end let $P_i= I - n_i\otimes n_i$ be the tangent projection associated 
with the exact domain $\Omega_i$ and let $P_M = I - n_M \otimes n_M$ be a constant projection 
such that 
\begin{equation}\label{eq:PM}
\| P_i - P_M \|_{L^\infty(M)} \lesssim h
\end{equation}
Since the normal field is smooth and $\text{diam}(M) \sim h$ we can construct $P_M$ using a constant approximation $n_M$ of $n$ on the macro element $M$. Adding and subtracting $P_M$ and using the 
triangle inequality we get 
\begin{align}
h \| \nabla_i v_M \|^2_{\mcE_{h,i}(\overline{M})} 
&= h \| P_i \nabla v_M \|^2_{\mcE_{h,i}(\overline{M})}  
\\
&\lesssim   h \| (P_i - P_M) \nabla v_M \|^2_{\mcE_{h,i}(\overline{M})}  
+ h \| P_M \nabla v_M \|^2_{\mcE_{h,i}(\overline{M})} 
\\ 
&=
II_1 + II_2
\end{align}
For $II_1$ we use (\ref{eq:PM}) and (\ref{eq:approx-est}) to conclude that 
\begin{align}
II_1 &=  h \| (P_i - P_M) \nabla v_M \|^2_{\mcE_{h,i}(\overline{M})}  
\\
&\lesssim   h^3 \|\nabla v_M \|^2_{\mcE_{h,i}(\overline{M})}  
\\
&
\lesssim  h \|\nabla v_M \|^2_{M}  
\\
&
\lesssim  h \|\nabla v_M \|^2_{T_M}  
\\
&\lesssim  h^2 \| \nabla v_M \|^2_{T_M\cap \Omega_i}
\\
&\lesssim h^2 \| \nabla (v_M - v ) \|^2_{M\cap \Omega_i} + h^2 \| \nabla v \|^2_{M\cap \Omega_i} 
\\
&\lesssim  h \| \nabla (v_M - v ) \|^2_{M} + h^2 \| \nabla_i v \|^2_{M\cap \Omega_i} +  h^2 \| \nabla_n v \|^2_{M\cap \Omega_i} 
\\ \label{eq:macro-var-II1}
&
\lesssim h \sum_{k=0}^1 h^{2k - 1} \|[ \nabla^k v ] \|^2_{\mcF_{h,i}(M)}  
+ h^2 \| \nabla_i v \|^2_{M\cap \Omega_i} +  h^2 \| \nabla_n v \|^2_{M\cap \Omega_i} 
\end{align}
Next for $II_2$ we pass from edges to the macro element using an inverse 
inequality, then using the fact that $P_M \nabla v_M$ is constant on $M$ 
we pass from $M$ to $T_M \cap \Omega_i$, then we add and subtract $P_i$ 
and employ (\ref{eq:PM}), 
\begin{align}
II_2 &= h \| P_M \nabla v_M \|^2_{\mcE_{h,i}(\overline{M})} 
\\
&
\lesssim h^{-1} \| P_M \nabla v_M \|^2_M
\\
&
\lesssim \| P_M \nabla v_M \|^2_{T_M \cap \Omega_i}
\\
&
\lesssim \| P_i \nabla v \|^2_{T_M \cap \Omega_i} + \| (P_M - P_i) \nabla v \|^2_{T_M \cap \Omega_i}
\\
&
\lesssim \| \nabla_i v \|^2_{M \cap \Omega_i} + h^2 \| \nabla v \|^2_{M \cap \Omega_i}
\\
&\lesssim 
(1 + h^2) \| \nabla_i v \|^2_{M \cap \Omega_i} + h^2 \| \nabla_{n} v \|^2_{M \cap \Omega_i}
\\ \label{eq:macro-var-II2}
&\lesssim \| \nabla_i v \|^2_{M \cap \Omega_i}  + h^2 \| \nabla_{n} v \|^2_{M \cap \Omega_i}
\end{align}
Together the bounds (\ref{eq:macro-vara-I}), (\ref{eq:macro-var-II1}), and (\ref{eq:macro-var-II2})  
of $I$, $II_1$ and $II_2$ give 
\begin{align}
I + II \lesssim \| \nabla_i v \|^2_{M \cap \Omega_i}  +  
 \sum_{k=0}^1 h^{2k - 2} \|[ \nabla^k v ] \|^2_{\mcF_{h,i}(M)}  
+ h^2 \| \nabla_{n} v \|^2_{M \cap \Omega_i}
\end{align}
which inserted into (\ref{eq:macro-var-a}) give
\begin{align}
h \| \nabla_i v \|_{A_i,\mcE_{h,i}}^2 &\lesssim 
\sum_{M\in \mcM_{h,i}} \| A_i \|_{L^\infty(M)} \| \nabla_i v \|^2_{M \cap \Omega_i}  
\\
&\qquad 
+  
\underbrace{\| A_i \|_{L^\infty(M)} \Big( \sum_{k=0}^1 h^{2k - 2} \|[ \nabla^k v ] \|^2_{\mcF_{h,i}(M)}  
+ h^2 \| \nabla_{n} v \|^2_{M \cap \Omega_i} \Big)}_{\sim \| v \|^2_{s_{h,i}}}
\\
&\lesssim 
\sum_{M\in \mcM_{h,i}} \underbrace{\| A_i \|_{L^\infty(M)} 
 \alpha_{i,M}^{-1}}_{c_{A_i,M}} \| \nabla_i v \|^2_{A_i, M\cap \Omega_i}  
 +  \| v \|^2_{s_{h,i,M}}
\end{align}
and thus the proof is complete.
\end{proof}

\begin{lem} \label{lem:ahsh-coer} If the stabilization parameters $\tau_{a_i}$,  see (\ref{eq:weightnorm}), are sufficiently large, then the forms $a_{h,i} + s_{h,i}$ satisfy the coercivity 
\begin{equation}\label{eq:ahsh-coer}
\tn v \tn^2_{h,i} \lesssim   a_{h,i} (v , v) + s_{h,i}(v,v)  \quad  v\in W_{h,i}, \quad i=0,1,2
\end{equation}
\end{lem}
\begin{proof} Using the definition of $a_{h,i}$ we get 
\begin{align}
a_{h,i} (v,v) + s_{h,i}(v,v) &= \|\nabla_i v \|^2_{A_i, \mcK_{h,i}} + \| v \|^2_{s_{h,i}}
\\
&\qquad 
- 2 (\langle \nu_i \cdot A_i \nabla_i v \rangle, [v])_{\mcE_{h,i}} + \tau_{a_i} h^{-1} \|\nu_i [v]\|^2_{A_i,\mcE_{h,i}}
\end{align}
Here we can estimate the third term on the right hand side using  Lemma \ref{lem:macro-var} together with the assumption $\sup_{M\in\mcM_{h,i}}  c_{ A_i, M} \lesssim 1$ to conclude that 
\begin{align}
  2(\langle \nu_i \cdot A_i \nabla_i v \rangle, [v])_{\mcE_{h,i}} & \leq 
  2h^{1/2}  \| \nabla_i v \|_{A_i, \mcE_{h,i}}  h^{-1/2} \|\nu_i [v] \|_{A_i, \mcE_{h,i}}
\\
&\leq
\delta h \| \nabla_i v \|^2_{A_i,\mcE_{h,i}}  + \delta^{-1} h^{-1} \|\nu_i [v]\|^2_{A_i, \mcE_{h,i}}
\\
&\leq 
\delta C \Big( \| \nabla_i v \|^2_{A_i,\mcK_{h,i}} +  \| v \|^2_{s_{h,i}} \Big) 
+ 
\delta^{-1} h^{-1} \|\nu_i [v]\|^2_{A_i, \mcE_{h,i}}
\end{align}
with $\delta>0$. We then have 
\begin{align}
a_{h,i} (v,v)+s_{h,i} (v,v) 
&\geq (1-C\delta) ( \|\nabla_i v \|^2_{A_i, \mcK_{h,i}} + \| v \|^2_{s_{h,i}} ) 
\\
&\qquad 
+ h^{-1}( \tau_{a_i} - \delta^{-1}) \|\nu_i [v]\|^2_{A_i,\mcE_{h,i}} 
\end{align}
where we may take $\delta$ small enough and $\tau_{a_i}$ large enough to obtain 
\begin{align}
a_{h,i} (v,v)+s_{h,i} (v,v) 
&\gtrsim  \|\nabla_i v \|^2_{A_i, \mcK_{h,i}} + \| v \|^2_{s_{h,i}} 
+ h^{-1} \|\nu_i [v]\|^2_{A_i,\mcE_{h,i}} 
\gtrsim \tn v \tn_{h,i}^2 
\end{align}
Here we 
at 
last used Lemma \ref{lem:macro-var} to control $h \| \langle  \nabla_i v_i \rangle \|^2_{A_i, \mcE_{h,i}}$.
\end{proof}

\begin{lem} \label{lem:cont-coer} The form $A_h$ is continuous 
\begin{equation}\label{eq:Ah-cont}
 A_{h} (v , w) \lesssim  \tn v \tn_{h,\bigstar}\tn w \tn_{h,\bigstar} \qquad  v,w\in \oplus_{i=0}^2 (H^1(\mcT_{h,i}) + W_{h,i})
\end{equation}
and if the stabilization parameters $\tau_{a_i}$, $i=0,1,2,$ defined in (\ref{eq:weightnorm}), are sufficiently large, then $A_h$ is coercive
\begin{equation}\label{eq:Ah-coer}
\tn v \tn^2_{h,\bigstar} \lesssim   A_{h} (v , v) \quad  v\in W_{h}
\end{equation}
\end{lem}
\begin{proof} The continuity (\ref{eq:Ah-cont}) follows directly from the Cauchy-Schwarz inequality 
and the definition (\ref{eq:dg-norm-star}) of the norm. 
To show the coercivity (\ref{eq:Ah-coer}) we first note that by construction 
$b_{h,i}(v,v)=\lambda_{b_i} \| [v]\|^2_{\mcE_{h,i}}$,
and therefore 
\begin{align}
A_h(v,v) &= \sum_{i=0}^2 \tkapi \Big( \underbrace{a_{h,i}(v_i,v_i) + s_{h,i}(v_i,v_i)}_{\gtrsim\tn v \tn^2_{h,i}} \Big) 
\\
&\qquad + \tkapi\lambda_{b_i} \| [v]\|^2_{\mcE_{h,i}} + \sum_{i=1}^2 \kappa_{0,i}^{-1} \|[\kappa v ]_i\|^2_{\partial \Omega_i \cap \Omega_0} 
\gtrsim \tn v \tn_h^2
\end{align}
where we at last used Lemma \ref{lem:ahsh-coer}. Finally, using Lemma \ref{lem:poincare-discrete-simple} we 
obtain the desired estimate 
\begin{align}
\tn v \tn_{h,\bigstar}^2 \lesssim  \tn v \tn_h^2 \lesssim A_h(v,v)
\end{align}

\end{proof}

\subsection{Interpolation Error Estimates}

To define the interpolant we need extensions of functions in $\Omega_i$ to $\mcT_{h,i}$. For the bulk domains, $\Omega_1$ 
and $\Omega_2$, the Stein extension theorem, see \cite{St70}, provides operators 
\begin{equation}
E_i: H^s(\Omega_i) \rightarrow H^s(\IR^d),\qquad i=1,2, \quad s\geq 0
\end{equation}
such that 
\begin{align}\label{eq:stab-ext}
\| E_i v \|_{H^s(\IR^d)} \lesssim \| v_i \|_{H^s(\Omega_i)}, \qquad   i=1,2, \quad s\geq 0
\end{align}
For the interface, $\Omega_0$, we construct an extension operator by composition with the closest point mapping
\begin{equation}
E_0: H^s(\Omega_0) \ni v \mapsto v \circ p \in H^s(U_{\delta_0}(\Omega_0))
\end{equation}
where we recall that $U_{\delta_0}(\Omega_0)\subset \Omega$ is an open tubular neighborhood of $\Omega_0$ of thickness 
$\delta_0>0$ and $p:U_{\delta_0}(\Omega_0) \rightarrow \Omega_0$, is the closest point mapping. We then have the stability 
\begin{align}\label{eq:ext-stab-0}
\| E_0 v \|_{H^s(U_\delta(\Omega_0) )} \lesssim \delta^{1/2} \| v_i \|_{H^s(\Omega_0)}, \qquad  \quad s\geq 0
\end{align}
Letting $U_{\delta} (\Omega_i ) = \cup_{x \in \Omega_i} B_\delta(x)$, where $B_\delta(x)$ is the open ball of diameter 
$\delta$ with center $x$ we may define the extension operator
\begin{equation}
E:\oplus_{i=0}^2 H^s(\Omega_i) \ni (v_0,v_1,v_2) \mapsto (E_0 v_0, E_1 v_1, E_2 v_2) \in \oplus_{i=0}^2 H^s(U_{\delta_0} (\Omega_i) )
\end{equation}

Let $\pi_{h,i}:L^2(\mcT_{h,i}) \rightarrow W_{h,i}$ be the Cl\'{e}ment 
interpolation operator. For all $v_i \in H^{2}(\mcT_{h,i})$ and $T\in
\mcT_{h,i}$ recall the following standard estimate 
\begin{equation}\label{eq:interpoltets}
\| v_i - \pi_{h,i} v_i \|_{H^m(T)} \lesssim h^{s-m} \| v_i \|_{H^s(\mcN_{h,i}(T))}
\qquad m\leq s \leq 2, \quad m=0,1, 2
\end{equation}
where $\mcN_{h,i}(T)\subset \mcT_{h,i}$ is the union of the elements in
$\mcT_{h,i}$ which share a node with $T$. In particular, we have the stability estimate 
\begin{equation}\label{eq:stabilityinterpol}
\| \pi_{h,i} v_i \|_{H^m(\mcT_{h,i})} \lesssim \| v_i \|_{H^m(\mcT_{h,i})}
\end{equation}
For $h\in (0,h_0]$ with $h_0$ small enough we have $\mcT_{h,i} \subset U_{\delta_0}(\Omega_i)$ and we may 
define interpolation operators by composing the Cl\'ement interpolation operator and the continuous extension operators. 
More precisely 
\begin{equation}
\pi_h:\oplus_{i=0}^2 L^2(\mcT_{h,i}) \ni (v_0,v_1,v_2) \mapsto (\pi_{h,0} E_0 v_0, \pi_{h,1} E_1 v_1, \pi_{h,2} E_2 v_2) \in  W_{h}
\end{equation}
Next we show an interpolation estimate in the dG norm. 
\begin{lem}\label{lem:interpolation} There is a constant such that for all $v \in \oplus_{i=0}^2 H^{2}(\Omega_i)$, 
\begin{equation}\label{eq:interpol-energy}
  \tn v^e- \pi_{h} v^e\tn_{h,\bigstar}^2 \lesssim h^2 \Big( \sum_{i=0}^2\| v_i \|^2_{H^{2}(\Omega_i)} \Big)
\end{equation}
\end{lem}

\begin{proof} Using the notation  $\rho_i = v_i - \pi_{h,i} v_i$ we have 
\begin{equation}
\tn \rho \tn_{h,\bigstar}^2 =\sum_{i=0}^2 \tn \rho_i \tn_{h,i,\bigstar}^2 + \sum_{i=1}^2 \kappa_{0,i}^{-1} \| [\kappa \rho ]_i \|^2_{\Omega_0}
\end{equation}
with 
\begin{align} 
\tn \rho_i \tn_{h,i,\bigstar}^2 
&\lesssim  \| \rho_i \|^2_{\mcK_{h,i}} +\| \nabla_i \rho_i \|^2_{\mcK_{h,i}} + h\| \langle \rho_i \rangle \|^2_{\mcE_{h,i}} 
\\
&\qquad 
+ h\| \langle \nabla_i \rho_i \rangle \|^2_{\mcE_{h,i}}+ h^{-1}\|[\rho_i]\|^2_{\mcE_{h,i}} + \| \rho_i \|^2_{s_{h,i}}
\end{align}
see (\ref{eq:dg-norm})--(\ref{eq:dg-norm-star}). To proceed with the estimates we first recall the following elementwise trace inequality that holds for elements $T \in \mcT_{h,i}$ in the bulk domain meshes $i=1,2,$ 
\begin{equation}\label{eq:trace-cut-bulk}
\| w \|^2_{\partial K } \lesssim h^{-1} \| w \|^2_{T}  + h \| \nabla w \|^2_T \qquad v \in H^1(T)
\end{equation}
where for $T \subset \Omega_i$ we have $K = T$ and therefore $\partial K = \partial T$ 
and for $T$ that intersect the interface $\Omega_0$ we have $\partial K = 
(\partial T \cap \Omega_i) \cup (T \cap \Omega_0)$,  see \cite{HaHaLa03} and \cite{HuWuXi17}. 

Using (\ref{eq:trace-cut-bulk}) followed by the interpolation error estimate (\ref{eq:interpoltets}) 
and the stability  (\ref{eq:stab-ext}) of the extension operator we obtain
\begin{align}
\tn \rho_i \tn^2_{h,i,\bigstar} &\lesssim h^{-2} \| \rho_i \|^2_{\mcT_{h,i}} +  \| \nabla \rho_i \|^2_{\mcT_{h,i}} + h^2  \| \nabla^2 \rho_i \|^2_{\mcT_{h,i}} 
\\
 &\qquad \lesssim h^2  \| E_i v_i \|^2_{H^2(\mcT_{h,i})}  \lesssim h^2  \| v_i \|^2_{\Omega_i}
\end{align} 
for $i=1,2.$ Next the interface term is estimated in a similar way
\begin{align}
&\kappa_{0,i}^{-1}\| [\kappa \rho ]_i \|^2_{\Omega_0} 
\lesssim 
\| \rho_0 \|^2_{\Omega_0} 
+ \| \rho_i \|^2_{\Omega_0} 
\lesssim 
\| \rho_0 \|^2_{\Omega_0} + h^{-1} \| \rho_i \|^2_{\mcT_{h,i}(\Omega_0)} 
+  h \| \nabla \rho_i \|^2_{\mcT_{h,i}(\Omega_0)}
\\
&\quad\lesssim 
\| \rho_0 \|^2_{\Omega_0} + h^3 \| E_i v_i \|^2_{H^2(\mcT_{h,i})} 
+  h^3 \| \nabla E_i v_i \|^2_{H^2(\mcT_{h,i})}
\lesssim  
\tn \rho_0 \tn^2_{h,0,\bigstar} + h^3 \| v_i \|^2_{H^2(\Omega_i)} 
\end{align}
where we used the triangle inequality and the trace inequality (\ref{eq:trace-cut-bulk}).

Finally, for the surface domain $\Omega_0$ we proceed in the same way but we instead employ the trace inequality 
\begin{align}
\| w \|^2_{\partial K } \lesssim \sum_{l=0}^2 h^{2 l - 4} \| \nabla^l w \|^2_{T}   \qquad w \in H^2(T)
\end{align}
see \cite{BurHanLarZah16}, and then the interpolation error estimate  (\ref{eq:interpoltets}) and the stability of the extension operator 
(\ref{eq:ext-stab-0}), where we first use the inclusion $ \mcN_h(\mcT_{h,0}) \subset U_\delta(\Omega_0)$ with $\delta \sim h$, 
\begin{align}
\tn \rho_0 \tn^2_{h,0,\bigstar} &\lesssim \sum_{l=0}^2 h^{2 l -3} \| \nabla^l \rho_0 \|^2_{\mcT_{h,0}} 
\lesssim \sum_{l=0}^2 h^{2 l -3}  h^{2(2-l)}  \| E_0 v_0 \|^2_{H^2(\mcT_{h,0})}  
\\
&
 \lesssim h \| E_0 v_0 \|^2_{H^2(U_\delta(\Omega_0))}  
 \lesssim  h \delta \| v_0 \|^2_{H^2(\Omega_0)}
 \lesssim  h^2  \| v_0 \|^2_{H^2(\Omega_0)} 
\end{align} 
which completes the proof.
\end{proof}

\subsection{A Priori Error Estimates}

\begin{thm}
  Let $u\in W\cap \bigoplus_{i=0}^2 H^2(\Omega_i)$ solve the convection diffusion problem \eqref{eq:BVP}-\eqref{eq:BVP-d} and $u_h \in W_h$ be the finite element solution defined by \eqref{eq:weakformuh}. Then, we have
  \begin{equation}
\tn u_h-u^e\tn_{h,\bigstar}  \lesssim  h \Big( \sum_{i=0}^2 \| u \|^2_{H^2(\Omega_i)} \Big)^{1/2}, 
\qquad 
\| u_h-u^e \|  \lesssim  h^2 \Big( \sum_{i=0}^2 \| u \|^2_{H^2(\Omega_i)} \Big)^{1/2}
  \end{equation}
\end{thm}
\begin{proof} 
Adding and subtracting the interpolant $\pi_h u$ and using the triangle inequality we get 
\begin{align}
\tn u - u_h \tn_{h,\bigstar} \leq \tn u - \pi_h u \tn_{h,\bigstar} + \tn \pi_h u - u_h \tn_{h,\bigstar} 
\end{align}
To estimate the second term we employ coercivity (\ref{eq:Ah-coer}), the linearity of $A_h$, the 
definition (\ref{eq:weakformuh}) of the dG method, the consistency (\ref{eq:consistency}), and finally 
the continuity (\ref{eq:Ah-cont}), as follows
\begin{align}
  \tn \pi_h u -u_h\tn_{h,\bigstar}^2 
  &\lesssim A_h(\pi_h u -u_h, \pi_h u -u_h)
\\  
  &=A_h(\pi_h u - u ,\pi_h u -u_h)+A_h(u-u_h,\pi_h u -u_h)
  \\
   &=A_h(\pi_h u - u ,\pi_h u -u_h)+\underbrace{A_h(u,\pi_h u -u_h) - L_h (\pi_h u -u_h)}_{=0} 
   \\
   &\lesssim \tn \pi_h u - u \tn_{h,\bigstar} \tn \pi_h u -u_h \tn_{h,\bigstar}
\end{align}
and thus 
\begin{equation}
  \tn \pi_h u -u_h\tn_{h,\bigstar} \lesssim \tn \pi_h u - u \tn_{h,\bigstar}
\end{equation}
We complete the proof by the interpolation result in Lemma~\ref{lem:interpolation}.  The proof of the $L^2$ estimate follows in the standard way using a duality argument.
\end{proof}

\section{Numerical Example}
We consider a similiar example as in~\cite{GrOlRe15}. The computational domain $\Omega$ is $[-1.5,1.5]\times [-1.5,1.5]$, the interface $\Omega_0$ is the unit circle, $\beta=(y,-x)$, $A_0=1$, $A_1=1$, $A_2=0.5$, $\kappa_1=2$, $\kappa_2=0.5$, $\kappa_{0,1}=1$, $\kappa_{0,2}=2$ and the source terms $f_i$, $i=0,1,2$ and the boundary data are taken such that the exact solution is
\begin{align}
  u_0(x,y)&=3x^2y-y^3 \\
  u_1(x,y)&=e^{1-x^2-y^2}u_0(x,y) \\
  u_2(x,y)&=2u_1(x,y)
  \end{align}
Note that $d_1=d_2=2$ and $d_0=1$. 

We approximate the interface $\Omega_0$ using a cubic spline parametrization, see~\cite{Zah17}. The proposed discontinuous CutFEM in Section~\ref{sec:cutfem} is used with stabilization parameters $\tau_{i,0}=A_i$, $\tau_{i,1}=0.1A_i$ for $i=1,2,$ (in equation~\eqref{eq:stab-bulk}) and $\tau_{0,0}=\tau_{0,1}=A_0$, $\tau_{0,2}=0.1A_0$ (in equation~\eqref{eq:stab-surface}). The resulting linear systems are solved by a direct solver. 

\begin{figure}\centering 
\includegraphics[width=0.38\textwidth]{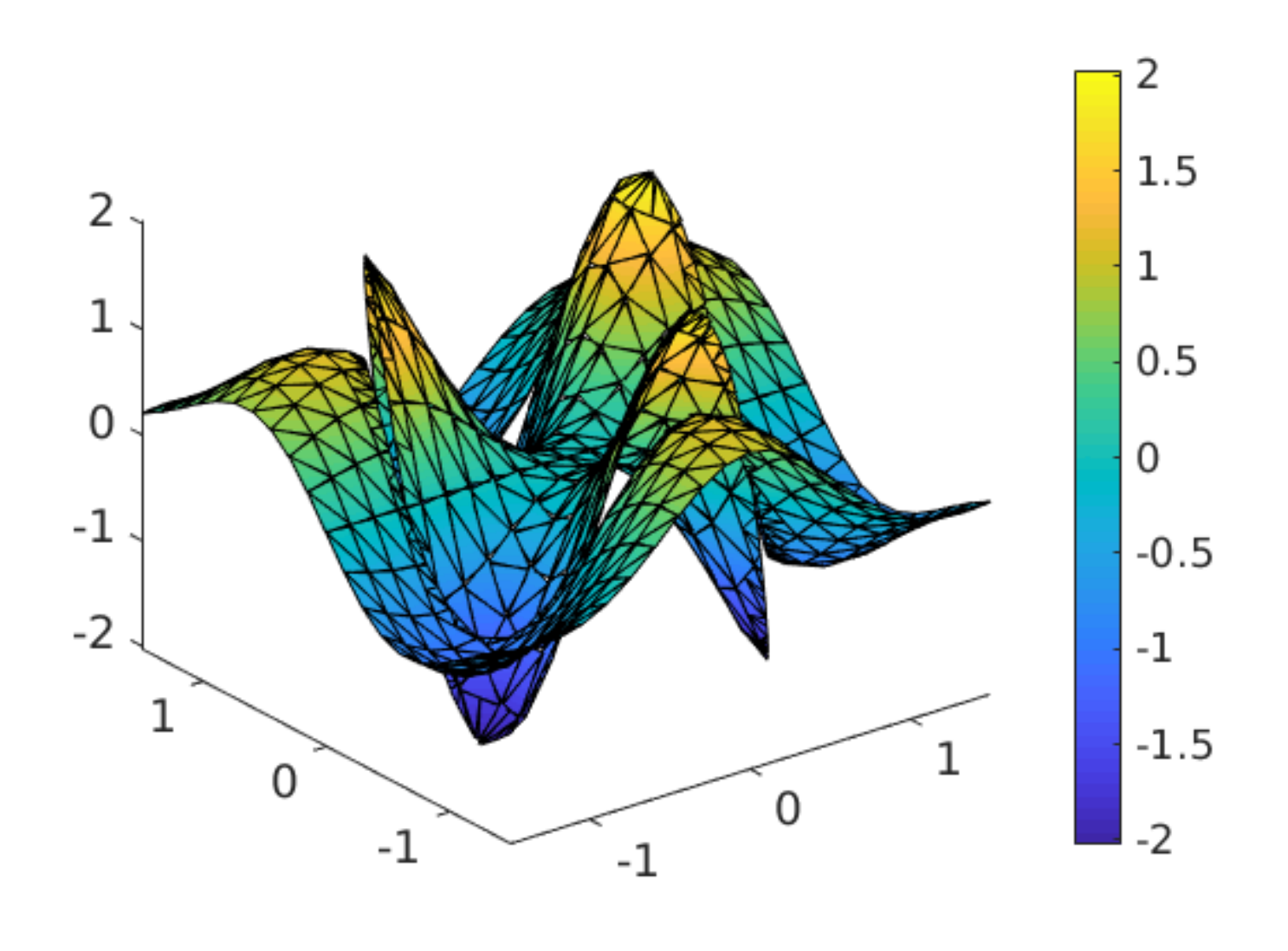} \hspace{0.5cm} 
\includegraphics[width=0.37\textwidth]{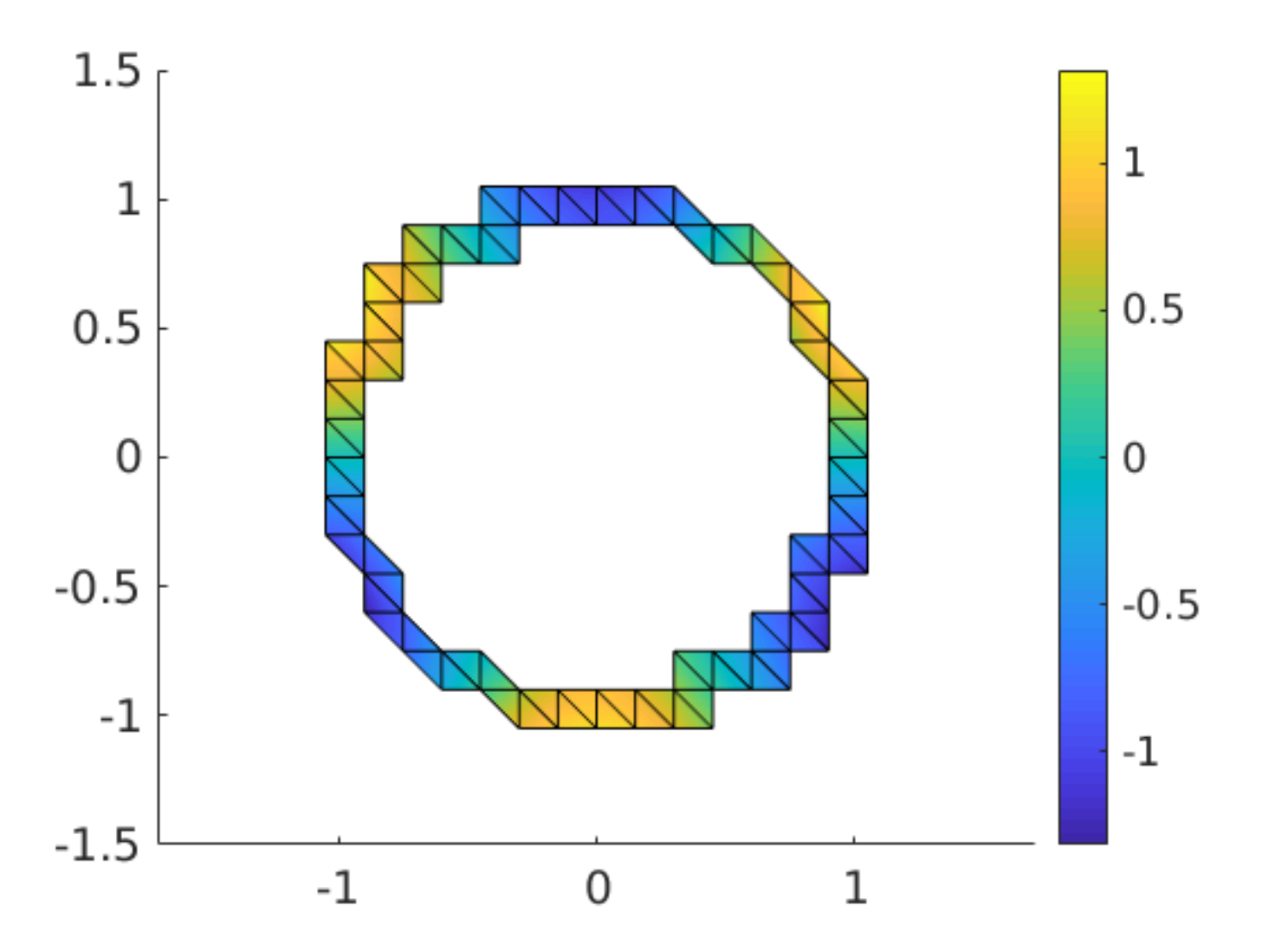} \\\vspace{0.5cm}
\caption{ Left panel: The numerical solution $(u_1,u_2)$ on the uniform background mesh with mesh size $h=0.15$. Right panel: The numerical solution $u_0$ on $\mcT_{h,0}$ with $h=0.15$. \label{fig:sol}}
\end{figure} 
The numerical solution on a uniform background mesh with mesh size $h=0.15$ is shown in Figure~\ref{fig:sol}. We used $\gamma_0=0.25$ and $\gamma_1=\gamma_2=0.125$ (see equation \eqref{eq:largeel}) in Algorithm 1. This resulted in a macro element partition with 20 edges in $\mcF^*_{h,0}$, 32 edges in $\mcF^*_{h,1}$, and 24 edges in $\mcF^*_{h,2}$. In case of full stabilization we would for this mesh instead apply stabilization on 90 edges when $i=0$, 138 edges when $i=1$, and 132 edges when $i=2$.  
We illustrate the difference between full stabilization and the macro element stabilization on the active mesh $\mcT_{h,1}$ for a courser mesh, $h=0.3$, in Figure~\ref{fig:illustfullvsmacrostab}.  We note that in the middle panel when $\gamma_1=0.5$ each cut element is marked as a small element in Algoritm 1 and always connected to an element that is entirely inside $\Omega_1$. However, also in this case when we apply stabilization following Algorithm 1 there are fewer edges (46 edges) on which stabilization is applied compared to using full stabilization (76 edges).   
\begin{figure}\centering
\includegraphics[width=0.3\textwidth]{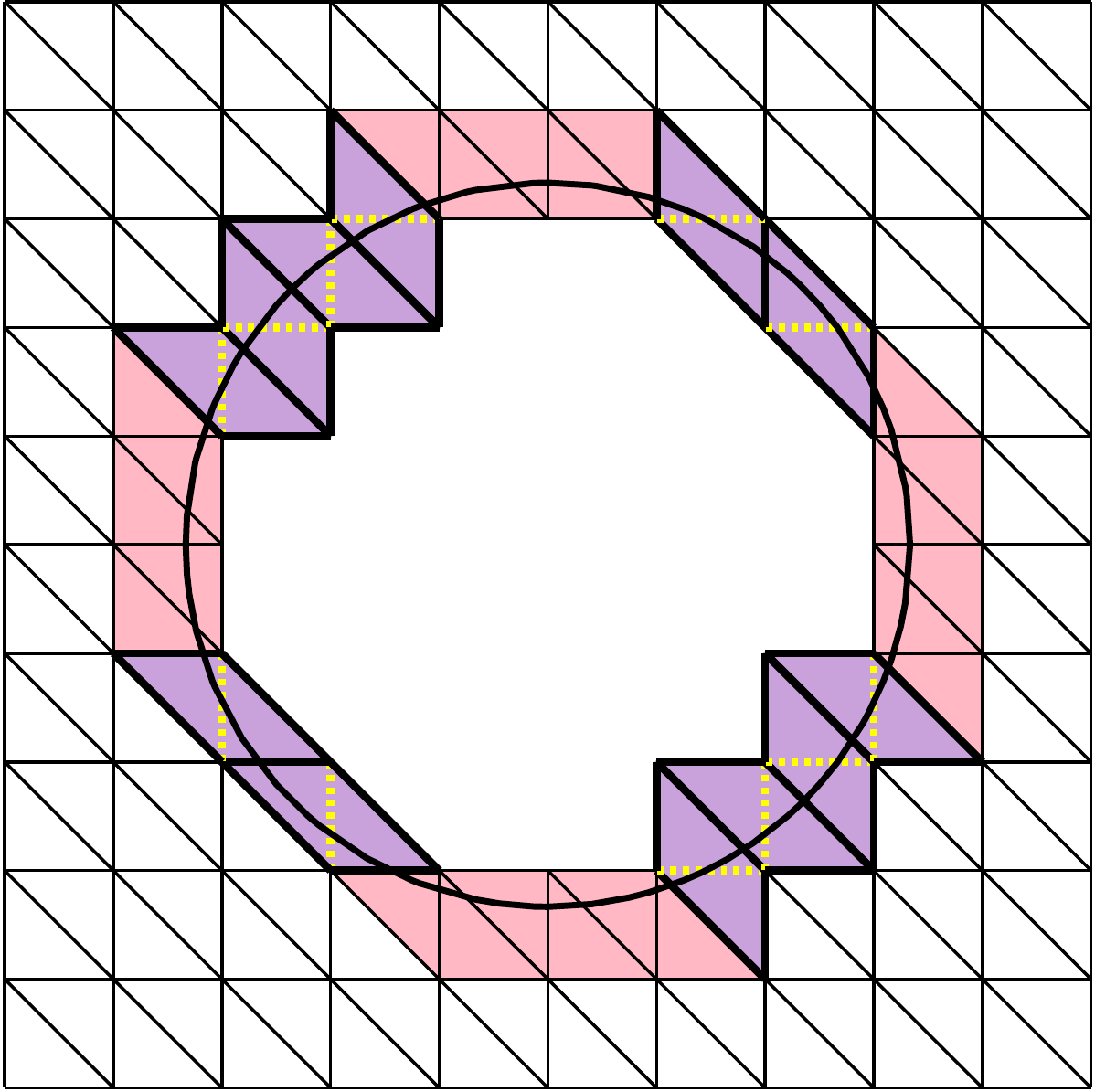} 
\includegraphics[width=0.3\textwidth]{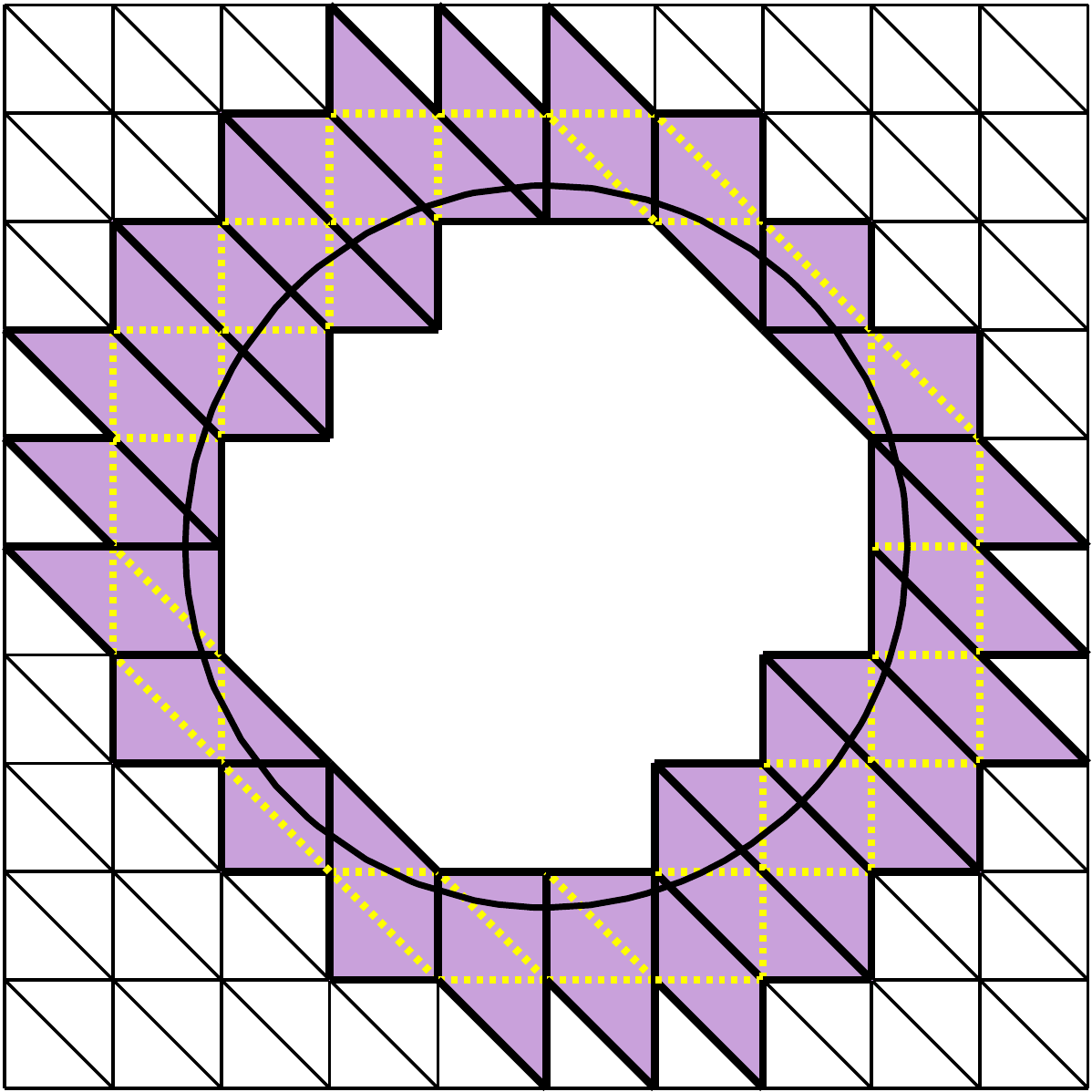} 
\includegraphics[width=0.3\textwidth]{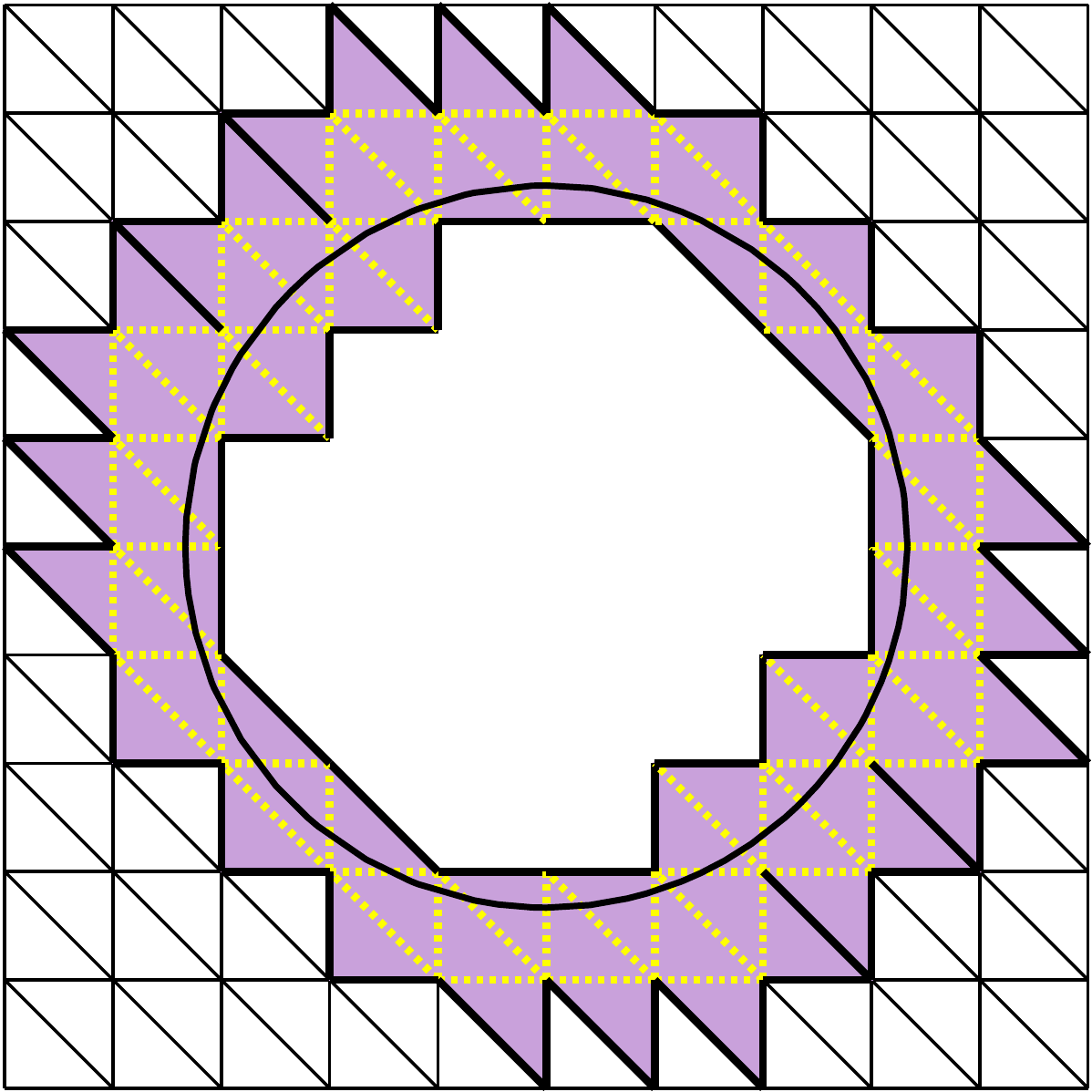} 
\caption{Macro elements in the active mesh $\mcT_{h,1}$ are shown in purple and edges on which stabilization is applied are shown with dotted lines for $\gamma_1=0.125$ (left panel), $\gamma_1=0.5$ (middle panel), and for full stabilization (right panel).  The mesh size is $h=0.3$. Stabilization is applied in total on 12 edges ($\gamma_1=0.125$), 46 edges ($\gamma_1=0.5$), and 72 edges (full stabilization).  \label{fig:illustfullvsmacrostab} }
\end{figure} 

For $\gamma_0=0.25$ and $\gamma_1=\gamma_2=0.125$ we show $L^2$- and $H^1$-errors in the bulk and the interface concentration for different mesh sizes in Figure~\ref{fig:erandcond_coupledP}. We obtain as expected first order convergence in the $H^1$-norm and second order convergence in the $L^2$-norm. We also show the spectral condition number of the scaled stiffness matrix associated with the form 
\begin{equation}\label{eq:Atilde}
\tilde{A}_h(v,w)=A_h((h^{1/2}v_0,v_1,v_2), (h^{1/2}w_0,w_1,w_2)).
\end{equation}
and we observe the expected $\mathcal{O}(h^{-2})$ behaviour. We refer to \cite{BurHanLarZah16} for an estimate of the condition number for a continuous CutFEM approximation of a coupled bulk-surface problem, with the same $h$-scaling as in $\tilde{A}_h$, 
that may be extended to the current discontinuous Galerkin method. In fact, combining the discrete Poincar\'e inequality in Lemma \ref{lem:poincare-discrete} and  Lemma \ref{lem:macro-control} we obtain a bound corresponding to 
Lemma 5.1 in \cite{BurHanLarZah16} and then the estimate of the condition number follows using the coercivity and continuity 
of the form $\tilde{A}_h$ as in Theorem 5.1 in \cite{BurHanLarZah16}. 

Note that for different mesh sizes the interface is positioned differently relative the background mesh but neither the errors or the condition numbers shown in Figure \ref{fig:erandcond_coupledP} are affected by how this relative position changes even though much less stabilization is applied than full stabilization.  
However, in Figure~\ref{fig:erandcond_coupledP_Dtol} where we compare results with different constants $\gamma_i$ we see that when $\gamma_i$ is too small both the error in the approximation of the interface concentration and the condition number of the resulting linear system are very sensitive to the position of the interface relative the background mesh. Errors both in $L^2$-norm and $H^1$-norm, and the condition number can become very large when enough stabilization is not applied. This is expected since coercivity of the form $A_h$ as well as the discrete Poincar\'{e} inequality do not hold in that case.  
The results in Figure~\ref{fig:erandcond_coupledP_Dtol} depend on how large the penalty parameters are chosen but here we have chosen the same penalty parameter independent of the parameters $\gamma_i$ in the stabilization terms $s_{h,i}$. Full stabilization give errors that are independent of the interface position relative the background mesh but we see that it also increases the magnitude of the errors.  The macro element stabilization with $\gamma_0=0.25$ and $\gamma_1=\gamma_2=0.125$ results in linear systems with the same condition number as using full stabilization but solutions with smaller errors. We observe similar behaviour for $H^1$-errors and therefore we only show the errors in $L^2$-norm in Figure~\ref{fig:erandcond_coupledP_Dtol}.

\begin{figure}\centering
\includegraphics[width=0.35\textwidth]{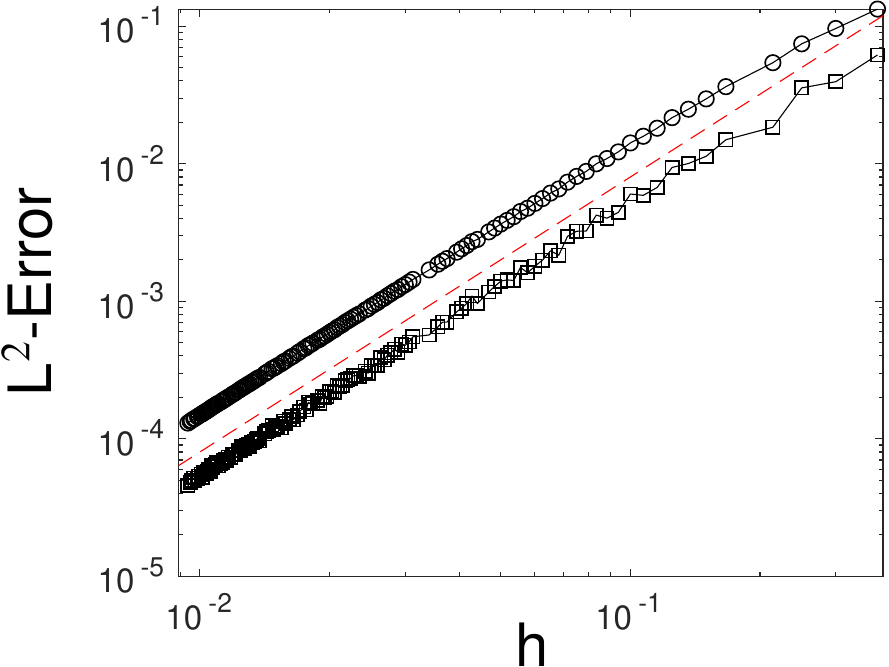} \hspace{0.5cm} 
\includegraphics[width=0.35\textwidth]{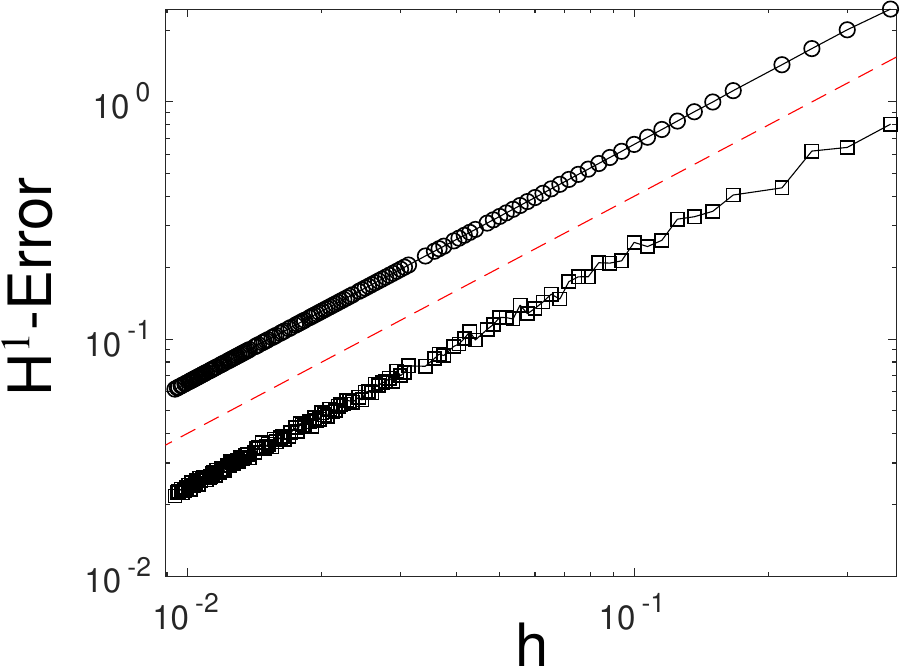} \\\vspace{0.5cm}
\includegraphics[width=0.35\textwidth]{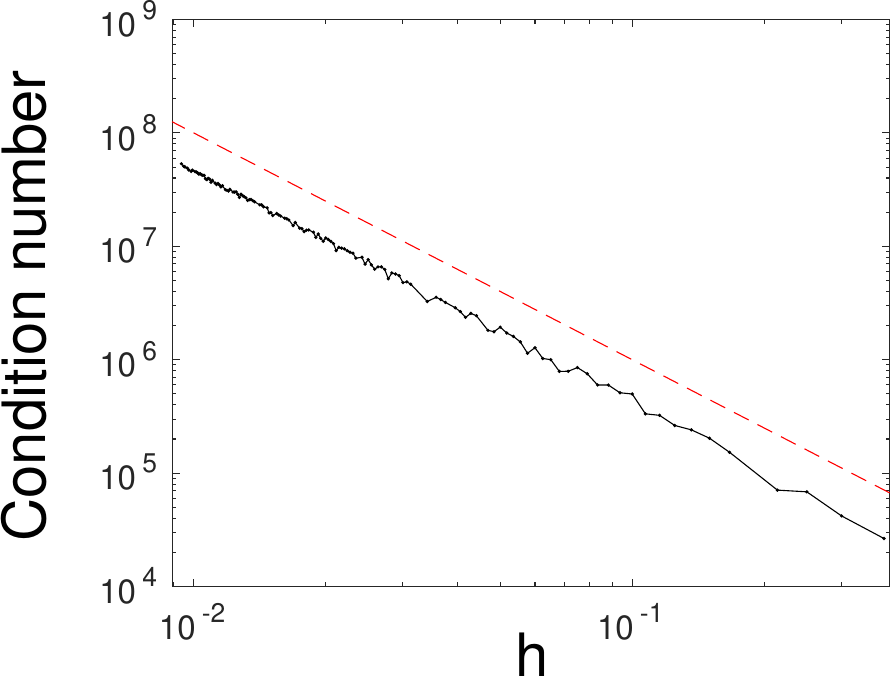} 
\caption{ The error and condition number versus mesh size $h$. Macro element stabilization with $\gamma_0=0.25$ and $\gamma_1=\gamma_2=0.125$. Circles represent the error in the bulk and squares represent the error on the interface. Top left panel: The error measured in the L$^2$-norm versus mesh size $h$.  The dashed line is $0.8h^2$.  Top right panel: The error measured in the $H^1$-norm versus mesh size $h$.  The dashed line is $4h$. Bottom: The spectral condition number of the scaled matrix associated with the form \eqref{eq:Atilde} versus mesh size $h$. The dashed line is $10^{4}h^{-2}$.\label{fig:erandcond_coupledP} }
\end{figure}

\begin{figure}\centering
\includegraphics[width=0.35\textwidth]{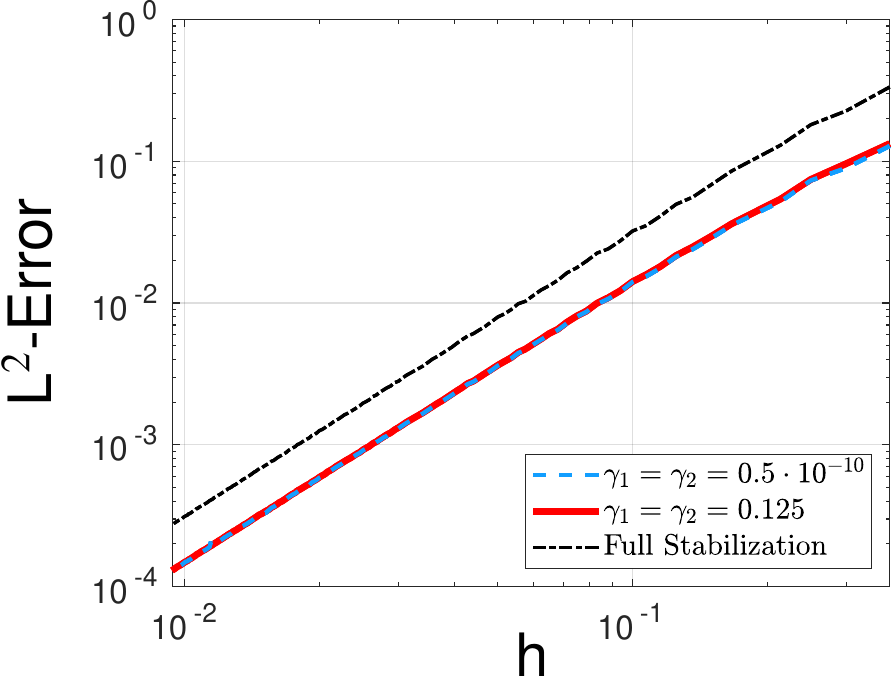} \hspace{0.5cm} 
\includegraphics[width=0.35\textwidth]{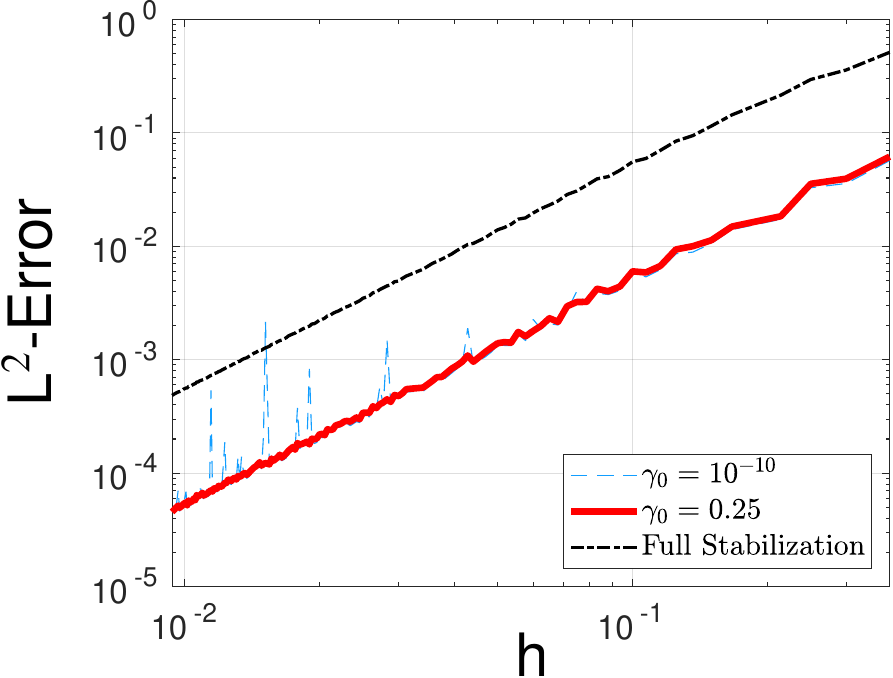} \\\vspace{0.5cm}
\includegraphics[width=0.35\textwidth]{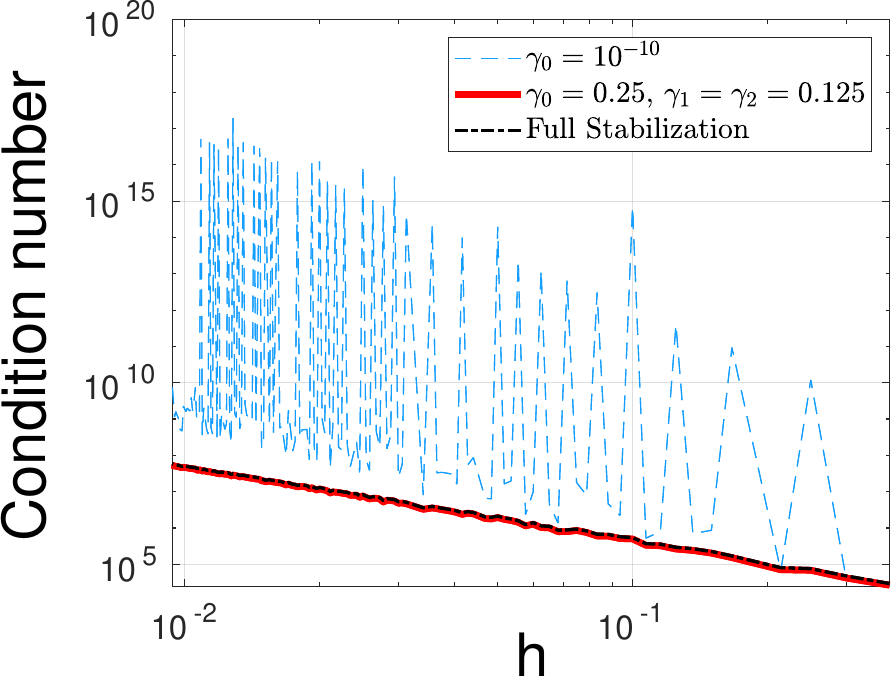} 
\caption{Errors measured in the $L^2$-norm and condition number versus mesh size $h$ for different constants $\gamma_i$ in the macro element stabilization (see equation \eqref{eq:largeel}). Top left panel: The error in the approximation of the bulk concentration. Top right panel: The error in the approximation of the interface concentration.  Bottom: The spectral condition number of the scaled matrix associated with the form \eqref{eq:Atilde} versus mesh size $h$.   \label{fig:erandcond_coupledP_Dtol}}
\end{figure} 

\section{Conclusions}
 
We have developed a stabilization method for discontinuous CutFEM approximations of coupled bulk-interface problems, which is 
localized to macro elements. This macro element stabilization approach leads to convenient proofs of basic stability results and also preserves the local conservation properties of the discontinuous Galerkin formulation on macro elements. Furthermore, the macro stabilization may be applied to continuous Galerkin methods as well as mass matrices and produces a diagonal block matrix with less coupling compared to standard full stabilization. We consider variable coefficients and diffusion as well as convection. The method enjoys optimal convergence and conditioning properties. Further developments include extension to higher order elements, time dependent 
domains, and improved implementations including efficient algorithms for computation of nearly optimal macro element partitions on surface domains.

\bigskip
\paragraph{Acknowledgement.} This research was supported in part by the Swedish Foundation for Strategic Research Grant No.\ AM13-0029, the Swedish Research Council Grants No. 2017-03911, 2018-05262, the Swedish strategic research programme eSSENCE, and the Wallenberg Academy Fellowship KAW 2019.0190.

\bigskip
\bigskip
\noindent
\footnotesize {\bf Authors' addresses:}

\smallskip
\noindent
Mats G. Larson,  \quad \hfill \addressumushort\\
{\tt mats.larson@umu.se}

\smallskip
\noindent
Sara Zahedi, \quad \hfill Mathematics, KTH, Sweden\\
{\tt sara.zahedi@math.kth.se}

\bibliographystyle{abbrv}
\footnotesize{
\bibliography{ref}
}
\end{document}